\documentclass[10pt]{amsart}
\usepackage{amscd,amssymb}
\usepackage{enumerate}
\usepackage[all]{xypic}
\usepackage{multirow}

\theoremstyle{plain}
\newtheorem{Thm}[equation]{Theorem}
\newtheorem{Cor}[equation]{Corollary}
\newtheorem{Prop}[equation]{Proposition}
\newtheorem{Lem}[equation]{Lemma}
\newtheorem{Rmk}[equation]{Remark}

\numberwithin{equation}{section}

\newcommand{\Hom}{\operatorname{Hom}}
\newcommand{\Gal}{\operatorname{Gal}}
\newcommand{\GL}{\operatorname{GL}}
\newcommand{\PGL}{\operatorname{PGL}}

\newcommand{\SL}{\operatorname{SL}}
\newcommand{\Sp}{\operatorname{Sp}}
\newcommand{\GSp}{\operatorname{GSp}}
\newcommand{\PGSp}{\operatorname{PGSp}}
\newcommand{\GSpin}{\operatorname{GSpin}}
\newcommand{\PGSpin}{\operatorname{PGSpin}}
\newcommand{\SO}{\operatorname{SO}}

\newcommand{\GSO}{\operatorname{GSO}}
\newcommand{\OO}{\operatorname{O}}

\newcommand{\JH}{\operatorname{JH}}
\newcommand{\Ind}{\operatorname{Ind}}
\newcommand{\Gm}{\mathbb{G}_m}
\newcommand{\simi}{\operatorname{sim}}

\newcommand{\Spin}{\operatorname{Spin}}
\newcommand{\C}{\mathbb C}

\newcommand{\Z}{\mathbb{Z}}
\newcommand{\R}{\mathbb{R}}

\newcommand{\bm}{\begin{multline*}}
\newcommand{\tu}{\end  {multline*}}

\newcommand{\ie}{{\em i.e. }}
\renewcommand{\L}{\mathcal{L}}

\title[LLC for $\Sp(4)$]{The local Langlands conjecture for $\Sp(4)$}
\author{Wee Teck Gan and Shuichiro Takeda}
\address{Mathematics Department, University of California, San Diego, 9500 Gilman Drive, La Jolla,
92093, U.S.A.}
\address{Department of Mathematics, Purdue University,
150 N. University Street, West Lafayette, IN 47907-2067, U.S.A.}
\email{wgan@math.ucsd.edu} \email{stakeda@math.purdue.edu}

\begin{document}

\maketitle

\begin{abstract}
We show that the local Langlands conjecture for $\Sp_{2n}(F)$
follows from that for $\GSp_{2n}(F)$, where $F$ is a non-archimedean local field of characteristic $0$. In particular, we prove the
local Langlands conjecture for $\Sp_4(F)$, based on our previous work
\cite{GT} on the local Langlands conjecture for
$\GSp_4(F)$. We also determine the possible sizes of
$L$-packets for $\Sp_4(F)$.
\end{abstract}


\section{\bf Introduction}


Let $F$ be a non-archimedean local field of characteristic $0$ and
residue characteristic $p$. Let $W_F$ be the Weil group of $F$ and
let $WD_F = W_F \times \SL_2(\C)$ be the Weil-Deligne group. It was
shown by Harris-Taylor \cite{HT} and Henniart \cite{He1} that there is a
natural bijection between the set of equivalence classes of
irreducible admissible representations of $\GL_n(F)$ and the set of
conjugacy classes of L-parameters for $\GL_n$, i.e. admissible
homomorphisms
\[
    \phi: W_F \times \SL_2(\C) \longrightarrow \GL_n(\C).
\]
This bijection satisfies a number of natural conditions which
determine it uniquely.\\

For a general connected reductive group $G$ over $F$, which we
assume to be split for simplicity, Langlands conjectures that there
is a surjective finite-to-one map from the set $\Pi(G)$ of
(equivalence classes of) irreducible admissible representations of
$G(F)$ to the set $\Phi(G)$ of (equivalence classes of)  admissible
homomorphisms
\[
    W_F \times \SL_2(\C) \longrightarrow  \hat{G},
\]
where $\hat{G}$ is the Langlands dual group of $G$ and the
homomorphisms are taken up to $\hat{G}$-conjugacy. This leads to a
partition of the set $\Phi(G)$ into a disjoint union of finite subsets,
which are the fibers of the map and are called $L$-packets.\\

In an earlier paper \cite{GT}, we have demonstrated the local Langlands conjecture for the group $\GSp_4$.
In this paper, we shall show how the local Langlands conjecture for
$\GSp_{2n}$, together with some of its expected properties, implies the
local Langlands conjecture for $\Sp_{2n}$. Together with \cite{GT}, this proves the
local Langlands conjecture for $\Sp_4$. Our main theorem is:\\
\vskip 10pt

\noindent{\bf \underline{Main Theorem}}\\
\noindent There is a surjective finite-to-one map
\[
    \mathcal{L}: \Pi(\Sp_4) \longrightarrow \Phi(\Sp_4)
\]
such that for an L-parameter $\varphi$, its fiber $\mathcal{L}_{\varphi}$ is of the same size (\ie the same cardinality) as the character group of the (abelian) component group
\[
    A_{\varphi} = \pi_0(Z(Im(\varphi))),
\]
where $Z(Im(\varphi))$ is the centralizer of the image of $\varphi$ in $\widehat{\Sp_4}= \SO_5(\C)$.
Moreover, the map $\mathcal{L}$ satisfies the following properties:
\vskip 5pt

\noindent (i) $\mathcal{L}$ preserves local $\gamma$-factors, $L$-factors and $\epsilon$-factors of pairs, whenever those factors are defined and satisfy some standard properties; see
\S \ref{S:properties} for the details. More precisely, suppose that $\sigma$ is an irreducible representation of $\GL_r(F)$ and $\varpi$ is an irreducible representation of $\Sp_4(F)$ which we assume to be  either generic or non-supercuspidal if $r > 1$. Let $\phi_{\sigma}$ and $\varphi_{\varpi}$ be the L-parameters of $\sigma$ and $\varpi$ respectively. Then we have
\[  \begin{cases}
\gamma(s, \varpi \times \sigma, std \otimes std, \psi) = \gamma(s, \varphi_{\varpi} \otimes \phi_{\sigma}, \psi) \\
L(s, \varpi \times \sigma, std \otimes std) = L(s, \varphi_{\varpi} \otimes \phi_{\sigma}) \\
\epsilon(s, \varpi \times \sigma, std \otimes std, \psi) = \epsilon(s, \varphi_{\varpi} \otimes \phi_{\sigma},\psi). \end{cases} \]
\noindent Here the functions on the RHS are the local factors of Artin type associated to the relevant representations of the Weil-Deligne group $WD_F$.
Those on the LHS are the local factors attached by Shahidi \cite{Shahidi} to the generic representations of
$\Sp_4(F) \times \GL_r(F)$ and the standard representation $std \otimes std$ of the dual group $\SO_5(\C) \times \GL_r(\C)$, and extended to all non-generic non-supercuspidal representations using the Langlands classification and multiplicativity. When $r = 1$, the local factors on the LHS can (alternatively) be defined for all representations using the doubling method of Piatetski-Shapiro and Rallis  \cite{LR}.

 \vskip 10pt

\noindent (ii) Suppose that $\varpi$ is a non-generic supercuspidal representation. For any irreducible supercuspidal representation $\sigma$ of $\GL_r(F)$, let $\mu(s, \varpi \boxtimes \sigma)$ denote the Plancherel measure associated to the family of induced representations
$I_P(\varpi \boxtimes \sigma,s)$ on $\Sp_{2r+4}(F)$, where we have regarded $\varpi \boxtimes \sigma$ as a representation of the Levi subgroup $\Sp_4(F) \times \GL_r(F)$. Then $\mu(s, \varpi \boxtimes \sigma)$ is equal to
\[  \gamma(s,  \varphi_{\varpi}^{\vee} \otimes \phi_{\sigma}, \psi) \cdot
\gamma(-s, \varphi_{\varpi} \otimes \phi_{\sigma}^{\vee},\overline{\psi}) \cdot
 \gamma(2s,  {\bigwedge}^2 \phi_{\sigma},\psi) \cdot
\gamma(-2s,   {\bigwedge}^2 \phi_{\sigma}^{\vee}, \overline{\psi}). \]
 \vskip 10pt

\noindent (iii) The representation $\varpi$ is a discrete series representation if and only if its L-parameter $\mathcal{L}(\varpi)$ does not factor through any proper parabolic subgroup of $\SO_5(\C)$.
\vskip 10pt

\noindent (iv) An L-packet $\mathcal{L}_{\varphi}$ contains a generic representation if and only if the adjoint L-factor $L(s , Ad \circ \varphi)$ is holomorphic at $s = 1$. Here, $Ad$ denotes the adjoint representation of $\SO_5(\C)$ on the complex Lie algebra $\mathfrak{so}(5)$. Moreover, $\mathcal{L}_{\varphi}$ contains a tempered generic representation if and only if $\varphi$ is a tempered L-parameter, i.e. $\varphi|_{W_F}$ has bounded image in $\SO_5(\C)$.
\vskip 10pt

\noindent (v) The map $\mathcal{L}: \Pi(\Sp_4) \longrightarrow \Phi(\Sp_4)$ is the unique one such that one has the following commutative diagram:
\[ \begin{CD}
\Pi(\GSp_4) @>L>>  \Phi(\GSp_4) \\
@VVV     @VVstdV  \\
\Pi(\Sp_4) @>\mathcal{L}>> \Phi(\Sp_4) \end{CD} \]
\vskip 5pt
\noindent  Here, note that the left vertical arrow is not a map at all: it is a correspondence defined by the subset of $\Pi(\GSp_4) \times \Pi(\Sp_4)$ consisting of pairs $(\pi, \varpi)$ such that $\varpi$ is a constituent of the  restriction of $\pi$ to $\Sp_4$. Moreover, the map $L$ is given by the local Langlands correspondence for $\GSp_4$ given in \cite[Main Theorem]{GT} and the right vertical arrow is defined by composition with the natural projection
\[  std : \GSp_4(\C)\rightarrow \PGSp_4(\C) \cong \SO_5(\C),  \]
which gives the standard representation of $\GSp_4(\C)$.

\vskip 10pt

\noindent (vi) The map $\mathcal{L}$ is uniquely determined by properties (i) (with $r \leq 3$), (ii) (with
$r \leq 4$) and (iii).
\vskip 20pt

Our proof is essentially a modification of the earlier work by Gelbart and Knapp \cite{GK}, in which they derived the local Langlands conjecture for $\SL_n$ from that for $\GL_n$. As suggested by property (v) of the Main Theorem, the definition of $\mathcal{L}$ is given as follows. For each
$\varpi \in\Pi(\Sp_4)$ one can find $\pi \in \Pi(\GSp_4)$ such
that the restriction $\pi|_{\Sp_4}$ contains $\varpi$. If
$\phi:WD_F\rightarrow\GSp_4(\C)$ is the $L$-parameter of
$\pi$ attached by \cite{GT}, then the composite
\[  \varphi:=std(\phi) \]
of $\phi$ with the projection $std$ gives the L-parameter for $\varpi$. After
checking that this is well-defined, the main point is to verify that the fiber $\mathcal{L}_{\varphi}$ of $\mathcal{L}$ over $\varphi$ has size $\# A_{\varphi}$. This is shown by establishing the following key short exact sequence:
\[  \begin{CD}
1 @>>> A_{\phi} @>>> A_{\varphi} @>>> I(\phi) @>>> 1, \end{CD} \]
where
\[  I(\phi) = \{ \text{characters $\chi$ of $W_F$: $\phi \otimes \chi \cong \phi$ as elements of $\Phi(\GSp_4)$} \}. \]
\vskip 15pt

We then give a more precise determination of the size of an $L$-packet of $\Sp_4$.
 Unlike the case of $\GSp_4$, where an $L$-packet is of size 1 or 2, for each $L$-parameter $\varphi\in\Phi(\Sp_4)$, we have:
\[
    \# L_{\varphi}=
    \begin{cases}
    1,2, 4\text{ or }8, &\text{ if $p\neq 2$;}\\
    1,2, 4,8\text{ or }16, &\text{ if $p= 2$}.
    \end{cases}
\]
In Section \ref{S:L-packet}, we give precise conditions for each of these possibilities in terms of the Galois theoretic  properties of a $\GSp_4$-parameter $\phi$ for which $\varphi = std(\phi)$. Our result is the analog of the well-known fact that an $L$-packet of $\SL_n$ associated to an irreducible representation $\pi$ of $\GL_2(F)$ has size 1, 2 or 4, depending on whether $\pi$ is primitive, dihedral with respect to a unique quadratic field, or dihedral with respect to 3 quadratic fields.
\vskip 5pt

\noindent{\bf Acknowledgments:} We thank Dipendra Prasad for several helpful conversations regarding symplectic parameters. W.T. Gan is partially supported by NSF grant 0500781. Also part of this paper was written when S. Takeda was at Ben Gurion University of the Negev in Israel. During his stay there, he was supported by the Skirball postodctoral fellowship of the Center of Advanced Studies in Mathematics at the Mathematics Department of Ben Gurion University.

\quad\\


\section{\bf From
$\GSp_{2n}$ to $\Sp_{2n}$}\label{S:Sp(2n)}


For any split connected reductive group $G$ over a non-archimedean local field $F$, let
$\Pi(G)$ be the set of equivalence classes of irreducible admissible
representations of $G(F)$, and $\Phi(G)$ the set of equivalence classes
of admissible homomorphisms $WD_F\rightarrow \hat{G}$. We will
derive the local Langlands conjecture for $\Sp_{2n}$ from that for
$\GSp_{2n}$. First let us note that
\[
    \widehat{\GSp_{2n}}=\GSpin_{2n+1}(\C), \quad \text{and} \quad
\widehat{\Sp_{2n}}=\PGSpin_{2n+1}(\C)\cong\SO_{2n+1}(\C).
\]
(We refer the reader to  \cite[Section 2]{AS} for the structure of the group $\GSpin$.)
Moreover, we have the canonical projection
\[
    std :\GSpin_{2n+1}(\C)\twoheadrightarrow\PGSpin_{2n+1}(\C) \cong \SO_{2n+1}(\C), \]
which is simply the standard $(2n+1)$-dimensional representation of $\GSpin_{2n+1}(\C)$.
We also let
\[
    \simi: \GSp_{2n}\rightarrow\Gm
\]
be the similitude character. Using $\simi$, we shall regard a character of $F^{\times}$
as a character of $\GSp_{2n}(F)$.
\vskip 10pt

For the rest of this section, we assume the following working
hypotheses, which we have verified for the case $n=2$ in \cite{GT}.
\vskip 15pt


\noindent{\bf \underline{Working Hypotheses}}
\vskip 10pt

\begin{enumerate}[(H1)]
\item There is a surjective finite-to-one map
$L:\Pi(\GSp_{2n})\rightarrow\Phi(\GSp_{2n})$ with the property that for each
$\phi\in\Phi(\GSp_{2n})$, the fibre $L_{\phi}$ is of the same size as
the component group
\[
    A_{\phi}=\pi_0(Z(Im(\phi))),
\]
where $Z(Im(\phi))$ is the centralizer of the image of $\phi$.
\vskip 5pt

\item For each quasicharacter $\chi$ of $F^{\times}$, we have
\[
	L(\pi\otimes \chi)=L(\pi) \otimes \chi,
\]
where on the right hand side, $\chi$ is viewed as a character of $WD_F$ via local class field theory.

\vskip 5pt

\item For each $\phi\in\Phi(\GSp_{2n})$, let
\[
    I(\phi)=\{\chi \in \widehat{(F^\times)}:
    \phi\cong\phi  \otimes \chi \}.
\]
Note that $I(\phi)$ consists of quadratic (and trivial) characters, since $\simi(\phi \otimes  \chi) =\simi(\phi) \cdot  \chi^2$.  Similarly, for each $\pi\in\Pi(\GSp_{2n})$, let
\[
    I(\pi)=\{\chi \in \widehat{(F^\times)}:
    \pi\cong\pi\otimes \chi  \},
\]
which consists also of quadratic characters.  Then for any $\pi\in L_{\phi}$,
\[
    I(\phi)=I(\pi).
\]
\end{enumerate}
\vskip 10pt

We also note the following:
\vskip 10pt

\begin{enumerate}[(H3a)]
\item For $\pi, \pi'\in L_{\phi}$, we have $I(\pi)=I(\pi')$.
\vskip 5pt

\item For each $\pi\in L_{\phi}$ and a quasicharacter $\chi$,
\[  \pi\otimes \chi \in L_{\phi} \Longrightarrow \pi\otimes \chi \cong\pi. \]
\end{enumerate}


\vskip 10pt

\noindent The following is easy to verify.
\begin{Lem}
Assume (H1) and (H2). Then (H3) holds if and only if (H3a) and (H3b) hold.
\end{Lem}
\vskip 10pt

\noindent Now we have the following:
\vskip 5pt

\begin{Prop}
If $n=2$, \ie for $\GSp_4$, all of the above working hypotheses are satisfied.
\end{Prop}
\begin{proof}
The hypotheses (H1) and (H2) are contained in the Main Theorem of \cite{GT}.
For the hypothesis (H3), note that (H2) already implies that $I(\pi) \subset I(\phi)$ if $\pi \in L_{\phi}$.
If $\chi \in I(\phi)$, then twisting by $\chi$ gives rise to a permutation of $L_{\phi}$. But $L_{\phi}$ has size $1$ or $2$, and in the latter case, $L_{\phi}$ contains a unique generic representation. Since twisting by $\chi$ preserves genericity, the induced permutation must be trivial, so that $\chi \in I(\pi)$ for any $\pi \in L_{\phi}$. This proves the proposition.
 \end{proof}
\vskip 10pt

For each $\pi\in\Pi(\GSp_{2n})$, we define
\[
    \JH(\pi):=\{\text{constituents of } {\pi}|_{\Sp_{2n}}\}.
\]
It has been shown by \cite{AD} that for each $\pi\in\Pi(\GSp_{2n})$, the restriction $\pi|_{\Sp_{2n}}$ is multiplicity free. Hence in the set $\JH(\pi)$, constituents can also be taken as constituents up to equivalence. Then the main result of this section is:
\vskip 5pt

\begin{Thm}\label{T:LLC_Sp}
Suppose that all of the above working hypotheses are satisfied (for example if $n=2$).
Then there is a surjective finite-to-one map
\[  \mathcal{L} :\Pi(\Sp_{2n})\rightarrow\Phi(\Sp_{2n}) \]
so that for each
$\varphi\in\Phi(\Sp_{2n})$, the fibre $\mathcal{L}_{\varphi}$ is given by
\[
    \mathcal{L}_{\varphi}=\bigcup_{\pi\in L_{\phi}}\JH(\pi) \quad\text{(disjoint
    union)},
\]
for some (and any) lift $\phi$ of $\varphi$, \ie such that
$std(\phi)=\varphi$. Moreover, the size of  $\mathcal{L}_{\varphi}$ is equal to the size of the component group
\[
    A_\varphi=\pi_0(Z(Im(\varphi))).  \]
\end{Thm}
\vskip 10pt

We will prove this theorem step by step. As we mentioned in the introduction, our proof is a modification of the work by Gelbart and Knapp \cite{GK}. So let us first quote a couple of general lemmas from \cite{GK} which we need for our proof.\\

\begin{Lem}\label{L:general_lemma1}
Let $G$ be a totally disconnected locally compact group and $H$ an open normal subgroup of $G$ such that $G/H$ is finite abelian. Also let $\varpi$ be an irreducible admissible representation of $H$. Suppose that $\pi$ and $\pi'$ are irreducible admissible representations of $G$ whose restrictions to $H$ are multiplicity free and contain $\varpi$. Then
\[
	\pi|_H\cong\pi'|_H
\]
and
\[
	\pi\cong \pi' \otimes \chi
\]
for some one dimensional character $\chi$ on $G$ which is trivial on $H$.
\end{Lem}
\begin{proof}
This is Lemma 2.4 of \cite{GK}.
\end{proof}
\vskip 5pt

\begin{Lem}\label{L:general_lemma2}
Let $G$ and $H$ be as in the above lemma, and $\pi$ an irreducible admissible representation of $G$. Assume that the restriction $\pi|_H$ is multiplicity free and written as
\[
	\pi|_H=\varpi_1\oplus\cdots\oplus\varpi_m,
\]
where each $\varpi_i$ is an irreducible admissible representation of $H$. Set
\[
	I_H(\pi)=\{\chi: G/H \longrightarrow \C^{\times}:  \pi\otimes\chi=\pi\}
\]
and
\[
	N_{\pi}=\bigcap_{\chi \in I_H(\pi)} \ker\chi,
\]
so that $I_H(\pi)$ is the Pontryagin dual of $G/N_{\pi}$.
Then $G/N_{\pi}$ acts simply transitively on the set $\{\varpi_1,\dots,\varpi_m\}$ of irreducible constituents of $\pi|_H$. In particular, we have
\[  m = \# I_H(\pi). \]
\end{Lem}
\begin{proof}
This is contained in Lemma 2.1 and  Corollary 2.2 of \cite{GK}. Note that our $I_H(\pi)$ is denoted by $X_H(\pi)$ there.
\end{proof}
\vskip 10pt

\noindent\underline{\textbf{Definition of the map $\mathcal{L}$}}\\

First we consider how a representation $\pi$ of $\GSp_{2n}(F)$ restricts to a representation of the group
\[
    H:=F^\times \cdot \Sp_{2n}(F).
\]
Note that $H$ is an open normal subgroup of $\GSp_{2n}(F)$ such that the quotient $\GSp_4(F)/H\cong F^\times/{F^\times}^2$ is a finite elementary abelian 2-group. As we have mentioned before, it is shown in \cite[Theorem 1.4]{AD} that for each $\pi\in\Pi(\GSp_{2n})$, the restriction
$\pi|_{\Sp_{2n}}$ is multiplicity free, and hence so is $\pi|_H$. Therefore we can apply Lemmas \ref{L:general_lemma1} and \ref{L:general_lemma2} for $G=\GSp_{2n}(F)$ and $H=F^\times \cdot \Sp_{2n}(F)$.\\

Let $\varpi$ be an irreducible admissible representation of $\Sp_{2n}(F)$ and $\mu$ a quasicharacter of $F^\times$. Then consider the representation $\mu\boxtimes\varpi$ of $F^\times\times\Sp_{2n}(F)$. If $\mu(-1)=\varpi(-1)$, then the representation factors through the surjection $F^\times\times\Sp_{2n}(F)\rightarrow H$. So we write a representation of $H$ as $\mu\boxtimes\varpi$, with the assumption that
$\mu(-1)=\varpi(-1)$.\\

Now for each $\varpi \in\Pi(\Sp_{2n})$, let us define
\[
    E(\varpi)=\{\pi\in\Pi(\GSp_{2n}): \pi|_{\Sp_{2n}} \text{contains }
    \varpi \},
\]
so that $E(\varpi)$ is the set of representations of $\GSp_{2n}(F)$ ``lying above'' $\varpi$. This set is non-empty, as can be seen as follows. For each given $\varpi$, let us pick a character $\mu$ of $F^\times$ such that $\mu(-1)=\varpi(-1)$. Then any irreducible constituent of $\Ind_{H}^{\GSp_{2n}(F)}\mu\boxtimes\varpi$ is an element of  $E(\varpi)$. We need the following.
\vskip 5pt

\begin{Lem}\label{L:same_packet}
If $\pi, \pi' \in E(\varpi)$, then $\pi'\cong\pi\otimes \chi$ for some character $\chi$.
\end{Lem}

\begin{proof}
Firstly, for each irreducible representation $\mu\boxtimes\varpi$ of $H$, let us define
\[
   E(\mu\boxtimes\varpi)=\{\pi\in\Pi(\GSp_{2n}): \text{$\pi|_H$ contains $\mu\boxtimes\varpi$} \}. \]
Secondly, assume that $\pi, \pi' \in E(\varpi)$. Then $\pi|_H$ (resp. $\pi'|_H$)
contains $\mu\boxtimes\varpi$ (resp. $\mu'\boxtimes\varpi$) for some
$\mu$ (resp. $\mu'$). So we have
\[ \pi\in E(\mu\boxtimes\varpi) \quad \text{and} \quad  \pi'\in
E(\mu'\boxtimes\varpi). \]
Then $\mu'/\mu$ is trivial on $\{\pm 1\}\subset F^\times$ because
\[ \mu(-1)=\mu'(-1)=\varpi(-1). \]
Hence $\mu'/\mu=\nu^2$ for some
quasicharacter $\nu$, and thus $\mu'\boxtimes \varpi \cong
(\mu\nu^2)\boxtimes\varpi$. Now consider
$\pi\otimes \nu |_H$. This contains
$(\mu\nu^2)\boxtimes\varpi \cong \mu'\boxtimes \varpi$ as a constituent since $\pi\in E(\mu\boxtimes\varpi)$, and so
$\pi\otimes \nu \in E(\mu'\boxtimes\varpi)$. Thus by Lemma \ref{L:general_lemma1},
$\pi'\cong\pi\otimes \nu\chi$ for some quasicharacter
$\chi$. This proves the lemma.
\end{proof}
\vskip 5pt

Now we can define the map
$\mathcal{L} :\Pi(\Sp_{2n})\rightarrow\Phi(\Sp_{2n})$ by
\[
    \mathcal{L}(\varpi)= std(L(\pi)), \text{ for any $\pi\in E(\varpi)$.}
\]
\vskip 5pt
\noindent This is well-defined. Indeed, if $\pi, \pi'\in E(\varpi)$, then the above
lemma implies that $\pi'\cong\pi\otimes \chi$ for some $\chi$,
but by (H2), $L(\pi')=L(\pi)\otimes \chi$ and so
\[
std(L(\pi'))= std(L(\pi)\otimes \chi)=std(L(\pi)). \]
\vskip 10pt

\noindent\underline{\textbf{Surjectivity of the map $\mathcal{L}$}}
\vskip 15pt

The surjectivity of  $\mathcal{L}$ follows from the following.
\vskip 5pt

\begin{Lem}
Any (continuous) projective representation of the Weil-Deligne group
\[
    \rho:W_F\times\SL_2(\C)\rightarrow \PGL_n(\C)
\]
can be lifted to a representation
\[
    \tilde{\rho}:W_F\times\SL_2(\C)\rightarrow \GL_n(\C).
\]
Moreover if $\tilde{\rho}'$ is another lift of $\rho$, then $\tilde{\rho}'=\tilde{\rho}\otimes\chi$ for some character $\chi$.
\end{Lem}
\vskip 5pt

\begin{proof}
Let $r_1=\rho|_{W_F}$ and $r_2=\rho|_{\SL_2(\C)}$ so that $r_1$ and $r_2$
are projective representations of $W_F$ and $\SL_2(\C)$,
respectively. Since $\SL_2(\C)$ is simply-connected, $r_2$ has a unique lift
\[  \tilde{r_2}:\SL_2(\C)\rightarrow \GL_n(\C).\]
Also Henniart \cite{He1} has shown that
$r_1$ has a lift
\[  \tilde{r_1}:W_F\rightarrow\GL_n(\C).\]
Now let us define
\[
    \tilde{\rho}:W_F\times\SL_2(\C)\rightarrow \GL_n(\C)
\]
by
\[
    \tilde{\rho}(g,h)=\tilde{r_1}(g)\tilde{r_2}(h).
\]
To show that $\tilde{\rho}$ is indeed a representation of $WD_F$, it suffices to show that $\tilde{r_1}(g)\tilde{r_2}(h)=\tilde{r_2}(h)\tilde{r_1}(g)$. For this, let us define
\[
    c:W_F\times\SL_2(\C)\rightarrow \GL_n(\C)
\]
by
\[
    c(g,h)=\tilde{r_1}(g) \cdot \tilde{r_2}(h)\cdot \tilde{r_1}(g)^{-1}\cdot \tilde{r_2}(h)^{-1}.
\]
\vskip 5pt

\noindent
We will show that $c(g,h)=1$ for all $g\in
W_F$ and $h\in\SL_2(\C)$. Note that
$c(g,h)$ is in the center of $\GL_n(\C)$ because the image of $c(g,h)$ in $\PGL_n(\C)$ is trivial. Also notice that, by taking the determinant
of $c(g,h)$, we see that $c(g,h)^n=1$. So the image of $c$ is in the
set of $n$-th roots of $1$. Now let us fix $g$. Then the map
$c(g,-)$ is a continuous map from $\SL_2(\C)$ to the $n$-th roots of
$1$. So $c(g,-)$ must be a constant map, since $\SL_2(\C)$ is connected. But
clearly $c(g,1)=1$, and so $c(g,h)=1$.
\vskip 5pt

The latter part of the lemma is obvious.
\end{proof}
\vskip 5pt

Hence we have

\begin{Prop}
Any $\varphi \in \Phi(\Sp_{2n})$ can be lifted to some $\phi \in \Phi(\GSp_{2n})$, and so $\mathcal{L}$ is surjective. Moreover if $\phi'$ is another lift of $\varphi$, then $\phi'=\phi\otimes  \chi$ for some character $\chi$ on $WD_F$.
\end{Prop}
\begin{proof}
The group $\GSpin_{2n+1}(\C)$ has a faithful $2^n$-dimensional Spin representation, so that one has a commutative diagram:
\[  \begin{CD}
\GSpin_{2n+1}(\C) @>>>  \GL_{2^n}(\C) \\
@VVV  @VVV \\
\text{PGSpin}_{2n+1}(\C)  @>>> \PGL_{2^n}(\C).  \end{CD} \]
The proposition is now immediate from the above lemma.
\end{proof}
\vskip 10pt

\noindent\underline{\textbf{The $L$-packet and the component group}}\\

Let us investigate the fibre $\mathcal{L}_{\varphi}$ for each
$\varphi \in \Phi(\Sp_{2n})$. We begin by showing that
\[
    \mathcal{L}_{\varphi}=\bigcup_{\pi\in
    L_{\phi}} \JH(\pi)\quad \text{(disjoint union)},
\]
where $\phi$ is any fixed lift of $\varphi$.
\vskip 5pt

First let us examine the RHS of the above equality. If $\phi$ and $\phi'$ are two different lifts of $\varphi$, then $\phi'=\phi \otimes \chi$ for some
character $\chi$ on $WD_F$. But (H2) implies that we have a
bijection $L_{\phi}\rightarrow L_{\phi'}$ given by $\pi\mapsto
\pi\otimes \chi$ and clearly
$\JH(\pi)=\JH(\pi\otimes \chi)$. Hence the union on the RHS is
independent of the choice of the lift $\phi$.
\vskip 5pt

Also we can see that the union is disjoint. For if $\pi, \pi'\in
L_{\phi}$ are such that there is a $\varpi \in\JH(\pi)\cap\JH(\pi')$,
then by Lemma \ref{L:same_packet}, we have
$\pi'\cong\pi\otimes \chi$ for some character
$\chi$. Then by (H3b), we must have $\pi'\cong\pi$.
 It is now quite easy to see that
\[  \mathcal{L}_{\varphi}=\bigcup_{\pi\in L_{\phi}} \JH(\pi). \]
Namely if $\varpi \in \mathcal{L}_{\varphi}$, then by definition of $\mathcal{L}_{\varphi}$, we must have $\varpi \in \JH(\pi)$ with $\pi\in L_{\phi}$ for some lift $\phi$. Conversely, if
$\varpi \in \JH(\pi)$ with $\pi\in L_{\phi}$ for some lift $\phi$,
then once again by the definition of $\mathcal{L}$, we see that
$\varpi \in\mathcal{L}_{\varphi}$.
\vskip 10pt

Finally, we show that the fibre $\mathcal{L}_{\varphi}$ has the same size as the component group
\[
    A_\varphi=\pi_0(Z(Im(\varphi)))=Z(Im(\varphi))/Z^\circ(Im(\varphi)),
\]
where $Z(Im(\varphi))$ is the centralizer of $Im(\varphi)$ and $Z^\circ(Im(\varphi))$ its connected component. By the above discussion, we see that
\[  \# \mathcal{L}_{\varphi} = \#L_{\phi} \cdot \# \JH(\pi), \]
where $\phi$ is a lift of $\varphi$ and $\pi \in L_{\phi}$.
Now Lemma \ref{L:general_lemma2} and hypothesis (H3) imply that
\[  \# \JH(\pi) = \# I(\pi) = \#I(\phi), \]
and (H1) says that
\[  \# L_{\phi} = \# A_{\phi} = \# Z(Im(\phi))/ Z^0(Im(\phi)). \]
Hence, we need to show that
\[   \# A_{\varphi} =  \#A_{\phi} \cdot \#I(\phi). \]
This is accomplished by the following key proposition:
\vskip 5pt

\begin{Prop}  \label{P:exact}
Let $\phi \in \Phi(\GSp_{2n})$ and set $\varphi = std(\phi)$. Then
there is a short exact sequence:
\vskip 5pt
\[  \begin{CD}
1 @>>> A_{\phi} @>\beta>> A_{\varphi} @>\alpha>> I(\phi) @>>> 1. \end{CD} \]
\end{Prop}
\vskip 5pt
Before giving the proof the proposition, let us quote a couple of basic facts about Lie groups, which we shall use for our proof.
\begin{Lem}\label{L:Lie}
Let $G$ be a Lie group and $H\subset G$ a closed subgroup. Then
\begin{enumerate}[(a)]
\item For each $s\in G/H$, there exists an open neighborhood $U$ of $s$ with a smooth (hence continuous) section $U\rightarrow G$, \ie there exist local smooth sections of $G/H$ in $G$.
\item If $H$ and $G/H$ are connected, then $G$ is connected.
\end{enumerate}
\end{Lem}
\begin{proof}
For (a) see \cite[Theorem 3.58(b)]{Warner}, and for (b) see \cite[Proposition 3.66]{Warner}.
\end{proof}
\vskip 5pt
Now we are ready to prove the proposition.
\begin{proof}[Proof of Proposition \ref{P:exact}]
The proof follows the lines of the proof of \cite[Theorem 4.3]{GK}.
We shall define  a group homomorphism
\[    Z(Im(\varphi))\rightarrow I(\phi),\quad s\mapsto \chi_s,
\]
as follows. First of all, let us define a
function $\chi_s:WD_F \rightarrow\GSpin_{2n+1}(\C)$ by
\[
    \chi_s(w)=\tilde{s} \cdot \phi(w) \cdot \tilde{s}^{-1} \cdot \phi(w)^{-1}\quad\text{for $w\in WD_F$,}
\]
where $\tilde{s}$ is a lift of $s$ to $\GSpin_{2n+1}(\C)$. It is clear that $\chi_s$ is
well-defined.  We now note the following properties:

\vskip 5pt

\begin{enumerate}[$\bullet$]
\item $\chi_s$ takes values in $\{ \pm 1 \}$ in the center of $\GSpin_{2n+1}(\C)$.
Indeed, since $std \circ \chi_s=1$, we have
$\chi_s(w)\in Z_{\GSpin_{2n+1}}=\C^\times$. On the other hand, $\chi_s(w)$ is a commutator and thus lies in the derived group $\Spin_{2n+1}(\C)$ of $\GSpin_{2n+1}(\C)$. Hence
\[  \chi_s(w) \in Z_{\GSpin_{2n+1}} \cap \Spin_{2n+1}(\C) = \{ \pm 1 \}. \]

\vskip 5pt

\item $\chi_s$ is a homomorphism, and hence by the above assertion a quadratic character of $WD_F$. This is because
\begin{align*}
    \chi_s(ww')&=\tilde{s} \cdot \phi(ww') \cdot \tilde{s}^{-1} \cdot \phi(ww')^{-1}\\
    &=\tilde{s} \cdot \phi(w)\cdot \tilde{s}^{-1} \cdot \left(\tilde{s}\phi(w')\tilde{s}^{-1}
    \phi(w')^{-1}\right)\phi(w)^{-1}\\
    &=\tilde{s} \cdot \phi(w) \cdot \tilde{s}^{-1}\cdot \phi(w)^{-1}\left(\tilde{s}\phi(w')\tilde{s}^{-1}
    \phi(w')^{-1}\right),
    \text{ because $\tilde{s}\cdot \phi(w')\cdot \tilde{s}^{-1}\cdot \phi(w')^{-1}\in\C^{\times}$}\\
    &=\chi_s(w) \cdot \chi_s(w').
\end{align*}
\vskip 5pt

\item $\chi_s \in I(\phi)$. Indeed, since
\[  \tilde{s} \cdot \phi(w) \cdot \tilde{s}^{-1}=\chi_s(w)\cdot \phi(w), \]
we have
\[  \phi\cong\phi\otimes \chi_s \quad \text{ as elements of $\Phi(\GSp_4)$,} \]
and thus  $\chi_s\in I(\phi)$. The assignment $s \mapsto \chi_s$ thus gives a map
\[  \alpha: Z(Im(\varphi)) \longrightarrow I(\phi). \]
\vskip 5pt

\item The map $\alpha$ is a group homomorphism.
This is proved in a similar way as the proof above that $\chi_s$ is a character.
\vskip 5pt

\item The homomorphism  $\alpha$ factors through $A_{\varphi}$, so that we have
\[  \alpha: A_{\varphi} \longrightarrow I(\phi). \]
To see this, note that for each $w \in WD_F$, the map $s \mapsto \chi_s(w)$ is a continuous homomorphism
\[   Z(Im(\varphi))\longrightarrow \{ \pm 1 \}, \]
because by Lemma \ref{L:Lie} (a) one can choose a local continuous section $s\mapsto\tilde{s}$, and thus it must be trivial on $Z^0(Im(\varphi))$.
\vskip 5pt

\item The homomorphism $\alpha$ is surjective.
This is similar to \cite[Theorem 4.3]{GK}. Namely, if $\chi \in I(\phi)$, then $\phi\cong\phi\otimes \chi$, \ie there exists $\tilde{s} \in \GSpin_{2n+1}(\C)$ such that
\[
	\tilde{s} \cdot \phi(w) \cdot \tilde{s}^{-1}=\chi(w) \cdot \phi(w)
\]
for all $w \in WD_F$. If we let $s=std(\tilde{s})$, we have $\chi_s=\chi$.
\end{enumerate}
  \vskip 10pt

We now examine the kernel of the map $\alpha$. Observe that, via the projection map $std$, one has a natural map
\[   \beta: Z(Im(\phi)) \longrightarrow Z(Im(\varphi)) \longrightarrow A_{\varphi}.  \]
The map $\beta$ has the following properties:
\vskip 5pt

\begin{enumerate}[$\bullet$]
\item The image of $\beta$ is precisely the kernel of $\alpha$. Indeed, if $\tilde{s}\in Z(Im(\phi))$, then
$\tilde{s} \cdot \phi(w) \cdot \tilde{s}^{-1}=\phi(w)$ for all $w\in
WD_F$,  so that
\[  \chi_{std(\tilde{s})}(w)=1 \quad \text{for all $w\in WD_F$.} \]
Hence the image of $\beta$ is contained in the kernel of $\alpha$. Conversely,
 let $s\in\ker \alpha$, so that $\chi_s(w)=1$
for all $w\in WD_F$. Then $\tilde{s} \cdot \phi(w) \cdot \tilde{s}^{-1}=\phi(w)$ for
a lift $\tilde{s}$ of $s$. So $\tilde{s} \in Z(Im(\varphi))$ and $\beta(\tilde{s}) = s$. This shows that the image of $\beta$ is precisely $\ker \alpha$.
 \vskip 10pt

\item The kernel of $\beta$ is equal to $Z^0(Im(\phi))$, so that $\beta$ induces an injection
 \[  \beta: A_{\phi} \longrightarrow A_{\varphi}. \]
 It is clear that $Z^0(Im(\phi)) \subset \ker\beta$, since $std$ maps $Z^0(Im(\phi))$ into $Z^0(Im(\varphi))$. It remains to show the reverse containment.
Let $\tilde{s}\in\ker\beta$, so that $std(\tilde{s})\in Z^0(Im(\varphi))$. Now observe that
\[  Z^0(Im(\phi))  \subseteq std^{-1}(Z^0(Im(\varphi)))  \subseteq Z(Im(\phi)).  \]
The first containment is what we have just observed. For the second, suppose that
\[  \text{$\tilde{s}\in std^{-1}(Z^0(Im(\varphi)))$ so that
$s = std(\tilde{s}) \in Z^0(Im(\varphi))$.} \]
Then we have seen that $\chi_s$ is trivial, so that
$\tilde{s}$ commutes with $Im(\phi)$, as desired. Note that since $\C^\times\subset std^{-1}(Z^0(Im(\varphi)))$ is closed and $std^{-1}(Z^0(Im(\varphi)))/\C^\times=Z^0(Im(\varphi))$ is connected, by Lemma \ref{L:Lie} (b) we know that $std^{-1}(Z^0(Im(\varphi)))$ is connected. Hence the first containment is an equality and so $\tilde{s} \in Z^0(Im(\phi))$.
\end{enumerate}
\vskip 5pt

The proposition is proved.
\end{proof}
\vskip 10pt

This completes the proof of
Theorem \ref{T:LLC_Sp}.\\


It is natural to ask if one can naturally index the elements of the L-packet $\mathcal{L}_{\varphi}$ by the Pontrjagin dual $\widehat{A_{\varphi}}$ of $A_{\varphi}$, say in the case of $\Sp_4$. It is well-known that to fix such a parametrization, one has to choose a generic character $\psi$ of the unipotent radical of a Borel subgroup of $\Sp_4(F)$.
Suppose that one has fixed such a choice of generic character $\psi$. Proposition \ref{P:exact} implies that
\[  \begin{CD}
1 @>>> \widehat{I(\phi)} @>>> \widehat{A_{\varphi}} @>>> \widehat{A_{\phi}} @>>> 1. \end{CD} \]
Moreover, we already know from \cite{GT} that $\widehat{A_{\phi}}$ naturally indexes the elements in the L-packet $L_{\phi}$ (which has size 1 or 2). Thus, it is completely natural to insist that for a character $\eta$ of $A_{\phi}$, the constituents of $\pi_{\eta}|_{\Sp_4}$ correspond to those elements of $\widehat{A_{\varphi}}$ which restrict to $\eta$. If $L_{\phi}$ contains a generic representation $\pi$ of $\GSp_4(F)$ (which corresponds to the trivial character of $A_{\phi}$), then we can further insist that the trivial character of $A_{\varphi}$ indexes the unique $\psi$-generic constituent in $\pi|_{\Sp_4}$.
Since $\widehat{I(\phi)} = \GSp_4(F)/ N_{\pi}$ acts transitively on the constituents of $\pi|_{\Sp_4}$,
we then obtain a parametrization of the constituents of  $\pi|_{\Sp_4}$ by the elements of $\widehat{A_{\varphi}}$ which restricts trivially to $A_{\phi}$.
\vskip 5pt

If $L_{\phi}$ contains another representation $\pi'$ (which is necessarily non-generic), then
we do not know how to parameterize the constituents of $\pi'|_{\Sp_4}$ by the remaining characters of
$A_{\varphi}$. Indeed, the set of constituents of $\pi'|_{\Sp_4}$ and the set of characters of $A_{\varphi}$ with non-trivial restriction to $A_{\phi}$ are both principal homogeneous spaces over $\widehat{I(\phi)}$,
and the choice of a base point in each will provide a bijection between them. However, it seems the only natural way to do this is via the character relations predicted by the theory of (twisted) endoscopy.
The recent work of Hiraga-Saito \cite{HS} on the local Langlands correspondence for inner forms of $\SL_n$
may shed some light on this matter.

 \vskip 10pt

\begin{Rmk}
Let us mention here that our derivation of the local Langlands conjecture for $\Sp_{2n}$ from that for $\GSp_{2n}$ works for other pairs of groups such as $(\GSO_n, \SO_n)$ and $(\GSpin_{2n}, \Spin_{2n})$ as long as the multiplicity free result of \cite{AD} is obtained for the relevant groups.
\end{Rmk}

\vskip 15pt


\section{\bf  Properties of $\mathcal{L}$}  \label{S:properties}


\vskip 5pt

To prove the Main Theorem in the introduction, it remains to verify the properties (i)-(vi) in the Main Theorem. In fact, many of these follow immediately from the construction of $\mathcal{L}$ and the analogous properties of the map $L$ for $\GSp_4$. We treat each property in turn:
\vskip 10pt

\begin{enumerate}[(i)]
\item Let us treat the preservation of local factors described in property (i).
We begin by recalling the precise definition of the local factors on the LHS of property (i).
Let $\varpi$ and $\sigma$ be irreducible admissible representations of $\Sp_4(F)$ and $\GL_r(F)$ respectively. If $\varpi$ and $\sigma$ are both generic, then  Shahidi has defined  in \cite[Theorem 3.5]{Shahidi} the local factors associated to $\varpi \boxtimes \sigma$ of $\Sp_4(F) \times \GL_r(F)$ and the standard representations of the dual group $\SO_5(\C) \times \GL_r(\C)$, which appears on the LHS of (i)  (by an analysis of the Plancherel measure treated in (ii) below).
One of the key properties of these local factors of Shahidi is that of multiplicativity of the local gamma factors under parabolic induction; for a discussion of this  property in the case of interest, the reader can consult \cite[Lecture 3, Pg. 318,  Example 1]{Sha2}. The L-factor and $\epsilon$-factor, on the other hand, only satisfy multiplicativity for the formation of standard modules.
\vskip 10pt

To extend the definition of the local factors beyond generic representations, we make use of the Langlands classification and multiplicativity.
If $\varpi \boxtimes \sigma$ is non-generic and non-supercuspidal, then by the Langlands classification, it is a quotient of a standard module $\text{Ind}_P^{\Sp_4 \times \GL_r} \tau$ with $\tau$ a
discrete series representation up to twist.  Since the Levi factor of any proper parabolic subgroup
of $\Sp_4 \times \GL_r$ is a product of $\GL_n$'s and perhaps $\SL_2$, $\tau$ is a generic representation (as any discrete series representation of $\GL_n$ or $\SL_2$ is generic). Thus, we may define the local factors associated to $\varpi \boxtimes \sigma$ by insisting that multiplicativity holds for such standard modules.  Again, we refer the reader to \cite[Pg. 318, Example 1]{Sha2} for the precise statements.
\vskip 10pt

In the case $r = 1$, there is an alternative definition of the local factors on the LHS of (i) via the doubling method of Piatetski-Shapiro and Rallis; a definitive treatment of this can be found in \cite{LR}.
Though it covers only the case $r = 1$, it has the advantage that the local factors can be defined for all representations $\varpi \boxtimes \sigma$, and not just for the generic ones. Moreover, these local factors also satisfy multiplicativity and agree with Shahidi's local factors when the representations involved are generic.
\vskip 10pt

With the above definitions,
the preservation of the local factors follows readily from \cite[Thm. 8.3]{GT} and \cite[Lemma 4.2]{GT}.

\vskip 10pt

\item  To establish property (ii) in the Main Theorem,  let us briefly recall the notion of the Plancherel measure in our context, and prove a few facts on the Plancherel measure necessary for our purposes.
\vskip 5pt

Let $\varpi$ be an irreducible admissible representation of $\Sp_4(F)$ and $\sigma$ a representation of $\GL_r(F)$, so that $\varpi\boxtimes\sigma$ is a representation of
\[
    M_r(F):=\Sp_4(F)\times\GL_r(F).
\]
Now $M_r$ is the Levi factor of a maximal parabolic subgroup $P_r=M_r\cdot N_r$ of $G_r=\Sp_{2r+4}$, so  that one can form the generalized principal series representation
\[
    I_{P_r}(s,\varpi\boxtimes\sigma)=\Ind_{P_r}^{G_r}\varpi\boxtimes\sigma|\det|^s\quad\text{(normalized induction)},
\]
where $\det$ is the determinant character of $\GL_r$. If $\bar{P_r}=M_r\cdot \bar{N_r}$ is the opposite parabolic, then we similarly have the induced representation $I_{\bar{P}_r}(s,\varpi\boxtimes\sigma)$. There is a standard intertwining operator
\[
    A_{\psi}(s,\varpi\boxtimes\sigma, N_r,\bar{N}_r):I_{P_r}(s,\varpi\boxtimes\sigma)\rightarrow I_{\bar{P}_r}(s,\varpi\boxtimes\sigma),
\]
defined by
\[
    A_{\psi}(s,\varpi\boxtimes\sigma, N_r,\bar{N}_r)f(g)=\int_{\bar{N}_r}f(\bar{n}g)\,d\bar{n}_{\psi}
\]
for $f\in I_{P_r}(s,\varpi\boxtimes\sigma)$. Then the composite $A_{\psi}(s,\varpi\boxtimes\sigma, \bar{N}_r,N_r)\circ A_{\psi}(s,\varpi\boxtimes\sigma, N_r,\bar{N}_r)$ is a scalar operator on $I_{P_r}(s,\varpi\boxtimes\sigma)$ for generic $s$ and the Plancherel measure is the scalar-valued meromorphic function defined by
\[
    \mu(s,\varpi\boxtimes\sigma, \psi)^{-1}=A(s,\varpi\boxtimes\sigma, \bar{N}_r,N_r)\circ A(s,\varpi\boxtimes\sigma, N_r,\bar{N}_r).
\]

By results of Shahidi \cite[Thm. 3.5]{Shahidi} and Henniart \cite{He4}, we have
\begin{Prop}\label{P:measure1}
Suppose that $\varpi\boxtimes\sigma$ is a generic representation of $M_r(F)=\Sp_4(F)\times\GL_r(F)$. Then, for appropriate measures $dn_{\psi}$ and $d \bar{n}_{\psi}$ on $N_r$ and $\bar{N}_r$ respectively,  $\mu(s,\varpi\boxtimes\sigma, \psi)$ is equal to
\begin{align*}
    &\gamma(s,\varpi^\vee\boxtimes\sigma,\psi)\cdot\gamma(-s,\varpi\boxtimes\sigma^\vee,
    \overline{\psi})
    \cdot\gamma(2s,\sigma,\bigwedge^2,\psi)\cdot\gamma(-2s,\sigma^\vee,\bigwedge^2, \overline{\psi})\\
    =&\gamma(s,\varphi_\varpi^\vee\otimes\phi_\sigma,\psi)\cdot\gamma(-s,\varphi_\varpi\otimes\phi_\sigma^\vee,\overline{\psi})
    \cdot\gamma(2s,\bigwedge^2 \phi_{\sigma},\psi)\cdot\gamma(-2s,\bigwedge^2 \phi_{\sigma}^{\vee},\overline{\psi}),
\end{align*}
where $\varphi_\varpi$ is the L-parameter for $\varpi$ by our construction and $\phi_\sigma$ is that for $\sigma$.
\end{Prop}
 \vskip 5pt

Now the same consideration can be applied to the group
\[
    M_r'(F)=\GSp_4(F)\times\GL_r(F),
\]
which is viewed as the Levi factor of a maximal parabolic subgroup $P'_r=M_r'\cdot N_r'$ of $G_r'=\GSp_{2r+4}$. Thus, for an irreducible admissible representation $\pi\boxtimes\sigma$, we have the principal series representations $I_{P_r'}(s,\pi\boxtimes\sigma)$ and $I_{\bar{P}_r'}(s,\pi\boxtimes\sigma)$, and the intertwining operator
\[
    A_{\psi}(s,\pi\boxtimes\sigma, N'_r,\bar{N'}_r):I_{P_r}(s,\pi\boxtimes\sigma)\rightarrow I_{\bar{P'}_r}(s,\pi\boxtimes\sigma),
\]
defined by
\[
    A_{\psi}(s,\pi\boxtimes\sigma, N'_r,\bar{N}'_r)f(g)=\int_{\bar{N}'_r}f(\bar{n}g)\,d\bar{n}_{\psi}.
\]
Then the Plancherel measure $\mu(s,\pi\boxtimes\sigma, \psi)$ is defined in the same way as the $\Sp_4$ case. Then we have
\vskip 5pt

\begin{Lem}
Let $\varpi$ and $\pi$ be irreducible admissible representations of $\Sp_4(F)$ and $\GSp_4(F)$, respectively, such that the restriction $\pi|_{\Sp_4(F)}$ contains $\varpi$ as a constituent. Then for any irreducible admissible representation $\sigma$ of $\GL_r(F)$, we have
\[
    \mu(s,\varpi\boxtimes\sigma, \psi)=\mu(s,\pi\boxtimes\sigma, \psi).
\]
\end{Lem}
\vskip 5pt

\begin{proof}
Let $\pi|_{\Sp_4(F)}=\oplus_i\varpi_i$, where each $\varpi_i$ is irreducible. Then we have
\[
    I_{P_r'}(s,\pi\boxtimes\sigma)|_{\Sp_{2r+4}(F)}=\bigoplus_i I_{P_r}(s,\varpi_i\boxtimes\sigma).
\]
Now looking at the integrals defining $A_{\psi}(s,\varpi_i\boxtimes\sigma, N_r,\bar{N}_r)$ and $A_{\psi}(s,\pi\boxtimes\sigma, N'_r,\bar{N'}_r)$, one immediately knows that
\[
    A_{\psi}(s,\pi\boxtimes\sigma, N'_r,\bar{N'}_r)|_{I_{P_r}(s,\varpi_i\boxtimes\sigma)}=A_{\psi}(s,\varpi_i\boxtimes\sigma, N_r,\bar{N}_r).
\]
Hence we have
\[
    \mu(s,\varpi_i\boxtimes\sigma, \psi)=\mu(s,\pi\boxtimes\sigma, \psi)
\]
for each $i$. Since $\varpi=\varpi_i$ for some $i$, the lemma follows.
\end{proof}
\vskip 5pt

Then we have:
\vskip 5pt

\begin{Prop}\label{P:measure2}
Let $\varpi$ and $\varpi'$ be in the same $L$-packet of our construction with $\varpi$ nongeneric supercuspidal, \ie $\mathcal{L}(\varpi)=\mathcal{L}(\varpi')$. Then for any supercuspidal representation $\sigma$ of $\GL_r(F)$ for any $r$, we have
\[
    \mu(s,\varpi\boxtimes\sigma,\psi)=\mu(s,\varpi'\boxtimes\sigma, \psi).
\]
\end{Prop}
\begin{proof}
 By our construction of $L$-packets, we can find irreducible admissible representations $\pi$ and $\pi'$ of $\GSp_4(F)$ so that $\varpi \subset \pi$,
$\varpi' \subset \pi'$ and $\pi$ and $\pi'$ are in the same $L$-packet (possibly $\pi=\pi'$) as in \cite{GT}. Then by  \cite[Thm. 9.6]{GT},  one knows that
\[
    \mu(s,\pi\boxtimes\sigma, \psi)=\mu(s,\pi'\boxtimes\sigma, \psi).
\]
Hence, the proposition follows by the previous lemma.
 \end{proof}

\vskip 10pt

Finally, the property (ii) of the Main Theorem follows from Props. \ref{P:measure1}
and \ref{P:measure2}.

 \vskip 10pt

\item If $\varpi$ is a constituent of $\pi|_{\Sp_4}$, then
$\varpi$ is discrete series if and only if $\pi$ is essentially discrete series. Similarly, $\varphi = std \circ \phi \in \Pi(\Sp_4)$ does not factor through any proper parabolic subgroup of $\SO_5(\C)$ if and only if $\phi$ does not factor through any proper parabolic subgroup of $\GSp_4(\C)$.
From these,  the property (iii) of the Main Theorem is an immediate consequence of \cite[Main Theorem (i)]{GT}.
\vskip 10pt

\item The property (iv) is a direct consequence of \cite[Main Theorem (vii)]{GT} and the definition of $\mathcal{L}$.

\vskip 10pt

\item The property (v) follows immediately by the construction of the  map $\mathcal{L}$; the uniqueness of $\mathcal{L}$ satisfying (v) is clear.
\end{enumerate}

The remaining property (vi), i.e. the characterization of the map $\mathcal{L}$, will be shown in the next section.

\vskip 15pt


\section{\bf Characterization of the Map $\mathcal{L}$}

In this section we show that our map $\mathcal{L}:\Pi(\Sp_4)\rightarrow\Phi(\Sp_4)$ is uniquely characterized by some of the properties of $\L$. Namely we prove
\begin{Thm}  \label{T:unique}
There is at most one map
\[
    \L:\Pi(\Sp_4)\rightarrow\Phi(\Sp_4)
\]
satisfying:
\begin{enumerate}[(a)]
\item $\varpi$ is a discrete series representation if and only if $\varphi_\varpi:=\L(\varpi)$ does not factor through any proper Levi subgroup of $\SO_5(\C)$.
\vskip 5pt

\item if $\varpi$ is generic or non-supercuspidal, then for any irreducible representation $\sigma$ of $\GL_r(F)$ with $r\leq 3$,
\[
    \begin{cases} L(s,\varpi\times\sigma)=L(s,\varphi_\varpi\otimes\phi_\sigma)\\
     \epsilon(s,\varpi\times\sigma, \psi)=\epsilon(s,\varphi_\varpi\otimes\phi_\sigma, \psi).\end{cases}
\]
\item if $\varpi$ is non-generic supercuspidal, then for any supercuspidal representation $\sigma$ of $\GL_r(F)$ with $r\leq 4$, the Plancherel measure $\mu(s,\varpi\boxtimes\sigma,\psi)$ is equal to
\[
    \gamma(s,\varphi_\varpi^\vee\otimes\phi_\sigma,\psi)\cdot\gamma(-s,\varphi_\varpi\otimes\phi_\sigma^\vee,\overline{\psi})
    \cdot\gamma(2s,\bigwedge^2\phi_{\sigma},\psi)\cdot\gamma(-2s,\bigwedge^2\phi_{\sigma}^{\vee} ,\overline{\psi}).
\]
\end{enumerate}
\end{Thm}

As one can see from the theorem, we have to resort to the Plancherel measure for the non-generic supercuspidal representations. This is due to the lack of a theory of the local factors for these representations. But as in the theorem, the Plancherel measure turns out to be sufficient to characterize the correspondence.
\vskip 10pt

To prove our theorem, we consider the two separate cases.\\

\vskip 5pt

\noindent\textbf{\underline{Case 1: $\varpi$ is generic or non-supercuspidal:}}\\

First we consider the case where $\varphi$ is generic or non-supercuspidal. The proof for this case is almost identical to the analogous case given in \cite[Thm. 10.1]{GT}; so we omit the details here.
\vskip 5pt

In fact, let us mention that recently it has been shown by the second author that by combining the results of Henniart \cite[Cor. 1.4 and Thm 1.7]{He3}, in which he characterized the local Langlands correspondence of $\GL_n$ by twists up to $\GL_{n-1}$, and  Chen's $n\times(n-2)$ local converse theorem for supercuspidal representations of $\GL_n$ \cite{Chen}, one can characterize the local Langlands conjecture of $\GL_n$ by twists only up to $\GL_{n-2}$. The proof of this result will appear elsewhere.

\vskip 10pt

\noindent\textbf{\underline{Case 2: $\varpi$ is non-generic supercuspidal:}}\\

Let $\pi$ be a non-generic supercuspidal representation of $\GSp_4(F)$ such that $\pi|_{\Sp_4(F)}$ contains $\varpi$ as a constituent. By \cite{GT}, we know that the L-parameter $\phi:=L(\pi)$ of $\pi$ is of the form $\phi_1\oplus\phi_2$ where each $\phi_i$ is a $2$-dimensional irreducible representation of the Weil-Deligne group $WD_F$ with $\det\phi_1=\det\phi_2$.
Now set
\[  \Phi := std (\phi)  = {\bf 1} \oplus( \phi_1^{\vee} \otimes\phi_2).
\]
Then we have shown in the previous section that the Plancherel measure $\mu(s,\varphi\boxtimes\sigma, \psi)$ is equal to
\[
   \gamma(s,\Phi^\vee\otimes\phi_\sigma,\psi)\cdot\gamma(-s,\Phi\otimes\phi_\sigma^\vee,
    \overline{\psi})
    \cdot\gamma(2s,\bigwedge^2\phi_{\sigma},\psi)\cdot\gamma(-2s,\bigwedge^2\phi_{\sigma}^{\vee},\overline{\psi}).
\]
Hence, if $\mathcal{L}$ is a map verifying the requirement (c), with $\varphi = \mathcal{L}(\varpi)$, then we have
\begin{equation} \label{E:Plancherel}
    \gamma(s,\Phi^\vee\otimes\phi_\sigma,\psi)\cdot\gamma(-s,\Phi\otimes\phi_\sigma^\vee,
    \overline{\psi})
   = \gamma(s,{\varphi}^\vee\otimes\phi_\sigma,\psi)\cdot\gamma(-s,\varphi\otimes\phi_\sigma^\vee,\overline{\psi}).
\end{equation}
We will show that, together with the requirement (a), this forces $\varphi =\Phi$, which completes the proof of the theorem.\\
\vskip 5pt

\noindent\textbf{Case I}:\\
First assume that $\phi$ and hence $\Phi$ is a representation of the Weil group $W_F$ (with the
$\SL_2(\C)$ in $WD_F$ acting trivially). Since $\phi$ is a discrete series parameter, $\Phi$ is a multiplicity free direct sum of irreducible representations, each of which is an orthogonal representation. Let us write $\Phi=\oplus_i\Phi_i$, where the $\Phi_i$'s are distinct. Now let us take $\phi_\sigma=\Phi_i$ for any fixed $i$. Then it is easy to see that the LHS of (\ref{E:Plancherel}) has a zero at $s=0$. Hence for some irreducible constituent $\varphi_i=\rho_i \boxtimes S_{r_i}$ of $\varphi$, the function
\[
    L(s,\varphi_i^\vee\otimes\Phi_i)\cdot L(s,\varphi_i\otimes\Phi_i^\vee)
\]
has a pole at $s=0$. This happens if and only if
\[
    \rho_i=\Phi_i\otimes|-|^{\pm(r_i-1)/2}.
\]
But  the requirement (a) implies that each irreducible constituent of $\varphi$ is an orthogonal representation, and so $\det\rho_i=\pm 1$. This implies that $r_i =1$ and so  $\varphi = \Phi$.
\vskip 10pt

\noindent\textbf{Case II}:\\
Next assume that $\phi_1=\chi\boxtimes S_2$ and $\phi_2$ is  a representation of the Weil group $W_F$ such that $\det\phi_2=\chi^2$. Then
\[
    \Phi={\bf 1} \oplus (\phi_1^{\vee}\otimes\phi_2)={\bf 1} \oplus (\chi^{-1}\cdot\phi_2\boxtimes S_2).
\]
So the LHS of (\ref{E:Plancherel}) with $\phi_\sigma=\chi^{-1}\cdot\phi_2$ becomes
\begin{align}\label{E:LHS}
    &\gamma(s,(\chi^{-1}\cdot\phi_2\boxtimes S_2)^\vee\otimes(\chi^{-1}\cdot\phi_2),\psi)\cdot
    \gamma(-s,(\chi^{-1}\cdot\phi_2\boxtimes S_2)\otimes  (\chi^{-1}\cdot\phi_2)^\vee,\psi)\notag\\
    &\qquad\qquad\cdot\gamma(s,(\chi^{-1}\cdot\phi_2)^\vee,\psi)\cdot\gamma(-s,\chi^{-1}\cdot\phi_2,\psi)\notag\\
    &=\text{($\epsilon$ factors)}\cdot
    \frac{\zeta(\frac{1}{2}+1-s)\cdot \zeta(\frac{1}{2}+1+s)}
    {\zeta(\frac{1}{2}+s)\zeta(\frac{1}{2}-s)},
\end{align}
which has a zero at $s=\frac{1}{2}$. Hence the RHS of (\ref{E:Plancherel}) with $\phi_\sigma=\chi^{-1}\cdot\phi_2$ must also have a zero at $s=\frac{1}{2}$, \ie $\varphi$ has a constituent
$\rho \boxtimes S_r$ such that
\begin{align*}
&\gamma(s, (\rho \boxtimes S_r)^\vee\otimes(\chi^{-1}\cdot\phi_2),\psi)\cdot
\gamma(-s, (\rho \boxtimes S_r)\otimes(\chi^{-1}\cdot\phi_2)^\vee,\psi)\\
&=\text{($\epsilon$ factors)}\cdot
\frac{L(\frac{r-1}{2}+1-s,\rho \otimes(\chi^{-1}\cdot\phi_2)^\vee)L(\frac{r-1}{2}+1+s,\rho^\vee\otimes(\chi^{-1}\cdot\phi_2))}
{L(\frac{r-1}{2}+s,\rho^\vee\otimes(\chi^{-1}\cdot\phi_2))L(\frac{r-1}{2}-s,\rho \otimes(\chi^{-1}\cdot\phi_2)^\vee)}
\end{align*}
has a zero at $s=\frac{1}{2}$, \ie the denominator of this fraction must have a pole at $s=\frac{1}{2}$. This happens if and only if
\[
    \rho =\chi^{-1}|-|^{r/2}\cdot\phi_2\quad\text{ or }\quad \chi^{-1}|-|^{-(r-2)/2}\cdot\phi_2.
\]
But once again the requirement (a) implies that $\det\rho= \pm 1$, which implies $r=0$ or $r=2$.  Since
 $r > 0$, we conclude that $r=2$ and $\rho =\chi^{-1}\cdot\phi_2$, which gives $\varphi = \mu \oplus(\chi^{-1}\cdot\phi_2\boxtimes S_2)$ for some 1-dimensional $\mu$. The fact that $\varphi$ takes value in $\SO_5(\C)$ implies that $\mu$ is trivial, so that $\varphi = \Phi$, as desired. \\

\vskip 5pt

\noindent\textbf{Case III}:\\
Finally, assume that $\phi_1=\chi\boxtimes S_2$ and $\phi_2=\mu\boxtimes S_2$ with $\chi^2=\mu^2$ and $\chi\neq\mu$ so that
\[
    \Phi={\bf 1} \oplus (\phi_1^{\vee} \otimes\phi_2)={\bf 1} \oplus \chi^{-1}\mu\oplus(\chi^{-1}\mu\boxtimes S_3).
\]
Then by setting $\phi_\sigma={\bf 1}$ and arguing as in Case II, one sees that the RHS has a zero at $s=0$, which implies that $\varphi$ contains
\[
    |-|^{\pm(t-1)/2}\boxtimes S_t
\]
as a constituent. But once again the requirement (a) implies that the determinant of this constituent is $\pm 1$, \ie $t=1$ and so $\varphi$ contains ${\bf 1}$. Similarly by taking $\phi_\sigma=\chi^{-1}\mu$ in (\ref{E:Plancherel}), the LHS is, up to $\epsilon$ factors, equal to
\begin{align}\label{E:RHS2}
    &\frac{L(1-s,\chi\mu^{-1})\cdot L(1+s,\chi\mu^{-1})\cdot\zeta(1-s)\cdot\zeta(1+s)\cdot\zeta(1-s+1)\cdot\zeta(1+s+1)}
    {L(s,\chi^{-1}\mu)\cdot L(-s,\chi^{-1}\mu)\cdot\zeta(s)\cdot\zeta(-s)\cdot\zeta(s+1)\cdot\zeta(-s+1)}\notag\\
    &=\frac{L(1-s,\chi\mu^{-1})\cdot L(1+s,\chi\mu^{-1})\cdot\zeta(2-s)\cdot\zeta(2+s)}
    {L(s,\chi^{-1}\mu)\cdot L(-s,\chi^{-1}\mu)\cdot\zeta(s)\cdot\zeta(-s)}.
\end{align}
Note that $\chi^{-1}\mu$ is a nontrivial quadratic character, so that $L(s, \chi^{-1}\mu)$ has no poles
on $\R$. So the above fraction has a zero at $s=0$ and a pole at $s=2$. The zero at $s=0$ implies that $\varphi$ contains
\[
    \chi^{-1}\mu|-|^{\pm(r-1)/2}\boxtimes S_r,
\]
and again the requirement (a) implies $r=1$, so that $\varphi$ contains $\chi^{-1}\mu$. Similarly, the pole at $s=2$ implies that it contains
\[
    \chi^{-1}\mu|-|^{-(q-3)/2}\boxtimes S_q\quad\text{ or }\quad\chi^{-1}\mu|-|^{(q-5)/2}\boxtimes S_q,
\]
and the requirement (a) gives $q=3$ for the former and $q=5$ for the latter. But for dimension reasons, the latter cannot occur here. So $\varphi$ contains $\chi^{-1}\mu\boxtimes S_3$. All these considerations imply that $\varphi=\Phi$.
\vskip 10pt

Theorem \ref{T:unique} is proved.

\vskip 15pt


\section{\bf Parameters of $\GSp_4$}  \label{S:GSp4}


The rest of the paper is devoted to the determination of the sizes of the L-packets of $\Sp_4$ in terms of Galois theoretic properties of their L-parameters. Before coming to that, it will be useful to
have a better understanding of the L-parameters of $\GSp_4$:
 \[  \phi: WD_F \longrightarrow \GSp_4(\C). \]
 We shall call such a $\phi$ a \textbf{symplectic parameter}.  In particular, we shall give
 a coarse classification of the most nondegenerate L-parameters of $\GSp_4$, namely those
\[  \phi: W_F \longrightarrow \GSp_4(\C) \]
which are irreducible as 4-dimensional representations of the Weil group $W_F$.
\vskip 5pt

For this purpose, let us recall that an irreducible representation $\phi$ of $W_F$ is called {\bf primitive} if it is not of the form $\Ind_{W_E}^{W_F} \sigma$ for a finite extension $E/F$, whereas $\phi$ is {\bf dihedral} with respect to a quadratic extension $E/F$ if $\phi=\Ind_{W_E}^{W_F} \sigma$ or equivalently if $\phi \otimes \omega_{E/F} \cong \phi$. It has been shown by Koch \cite{Ko} that a primitive representation exists only when $p$ divides $\dim\phi$.\\

\vskip 5pt

The main result of this section is the trichotomy of the following proposition.

\vskip 5pt

\begin{Prop}  \label{P:parameter}
Let $\phi: W_F \longrightarrow \GSp_4(\C)$ be an irreducible 4-dimensional representation with similitude character $\simi(\phi)$. Then we have the following possibilities:

\vskip 5pt

\begin{enumerate}[\upshape (I)]
\item $\phi$ is primitive. In this case, the 5-dimensional representation $std(\phi)$ is irreducible.

\item There is a quadratic extension $E/F$ (with $\Gal(E/F) = \langle \tau \rangle$), a primitive representation $\sigma$ of $W_E$ and a character $\chi$ of $W_E$ such that
\[
	\phi = \Ind_{W_E}^{W_F} \sigma,
	\quad \sigma^{\tau} \cong \sigma \cdot  \chi, \quad \chi^2 \ne 1 \quad \text{and}
	\quad \simi(\phi)|_{W_E} = \chi \cdot \det \sigma \ne \det \sigma.
\]
Moreover, the 5-dimensional representation $std(\phi)$ is reducible but does not have a 1-dimensional constituent; it decomposes as the sum of a 2-dimensional irreducible constituent and a 3-dimensional irreducible constituent.

\item There is a quadratic extension $E/F$ and an irreducible two-dimensional representation $\sigma$ of $W_E$ such that
\[
	\phi = \Ind_{W_E}^{W_F} \sigma \quad \text{and} \quad \simi(\phi)|_{W_E} = \det \sigma.
\]
In this case, the 5-dimensional standard representation $std(\phi)$ contains at least  one 1-dimensional constituent. Indeed, it contains $\omega_{E/F}$.
\end{enumerate}

The three situations above are mutually exclusive. When $p\ne 2$, only {\upshape (III)} can occur, but when $p =2$, all the three situations can occur.
\end{Prop}
\vskip 10pt

The rest of the section is devoted to the proof of the proposition.
We first study the primitive $\phi$'s. The following lemma proves the characterization of primitive $\phi$'s given in (I) of the proposition.
\vskip 10pt

\begin{Lem}  \label{L:primitive}
Let $\phi:W_F \longrightarrow \GSp_4(\C)$ be an irreducible representation.   Then
$\phi$ is primitive if and only if $std (\phi)$ is irreducible as a 5-dimensional representation.
\end{Lem}

\begin{proof}
Suppose that $\phi$ is primitive but $std (\phi)$ is reducible. Then the image of
$std (\phi)$ must be contained in the subgroup  $S(\OO_2(\C) \times \OO_3(\C))$ or $S(\OO_1(\C) \times \OO_4(\C))$ of $\SO_5(\C)$.
The preimage of each of these two groups in $\GSp_4(\C)$ is, respectively,  the normalizer of the Levi subgroup $\GL_2(\C) \times \GL_1(\C)$ of  a Siegel parabolic  or the normalizer of the subgroup $(\GSp_2(\C) \times \GSp_2(\C))^0$. In either case, the preimage is disconnected with two connected components, and its identity component acts reducibly.
Thus, $\phi$ becomes reducible when restricted to a subgroup of index two, which contradicts the assumption that $\phi$ is primitive. Thus, $std (\phi)$ must be irreducible if $\phi$ is primitive.
\vskip 10pt

Conversely, suppose that $std (\phi)$ is irreducible. We need to rule out the possibility that $\phi$ is induced.
Observe that as representations of $W_F$,
\[  ({\bigwedge}^2 \phi) \otimes \simi(\phi)^{-1} =  std (\phi)  \oplus \C. \]
Now we show that $\phi$ cannot be equal to $\Ind_{W_K}^{W_F} \chi$ with $K/F$ quartic. Indeed,
if $\phi$ has this form,  the restriction of $\phi$ to $W_K$ contains the 1-dimensional submodule $\chi$. It must then contain (at least) two 1-dimensional submodules, since it preserves a nondegenerate symplectic form up to scaling. From this, it follows that $\wedge^2 \phi|_{W_K}$ contains at least two 1-dimensional  submodules. This would imply that $std (\phi)|_{W_K}$ contains a 1-dimensional submodule and by Frobenius reciprocity, one would obtain a nonzero $W_F$-intertwining map from $std (\phi)$ to a 4-dimensional representation, which is a contradiction.
On the other hand, suppose that $\phi=\Ind_{W_E}^{W_F} \rho$ with $E/F$ quadratic. Then the restriction of $\phi$ to $W_E$ is the sum of two 2-dimensional submodules. Again, this implies that
$std(\phi)|_{W_E}$ contains a 1-dimensional submodule which is again impossible.
We have thus shown that $\phi$ is primitive.
\end{proof}

\vskip 10pt
Next, we  describe a construction of primitive $\phi$'s which was shown to us by D. Prasad.
\vskip 5pt

\begin{Prop}  \label{P:dipendra}
When the residue characteristic $p$ of $F$ is equal to $2$, there exists irreducible primitive representations $\phi: W_F \longrightarrow \GSp_4(\C)$.
\end{Prop}
\vskip 5pt

\begin{proof}
Suppose that  $\varphi^{\#}$ is an irreducible self-dual 5-dimensional representation of $W_F$; we shall show below that such a representation exists.  Such a $\varphi^{\#}$ must necessarily preserve a quadratic form and thus $\varphi^{\#}$ factors through $\OO_5(\C)$. By twisting by a quadratic character if necessary, we can ensure that $\varphi^{\#}$ factors through $\SO_5(\C)$.
 As we saw earlier,  such a $\varphi^{\#}$ admits a lifting
\[  \phi : W_F \longrightarrow \GSp_4(\C), \]
so that $std(\phi) = \varphi^{\#}$.
By the previous lemma, we know that $\phi$ must be primitive.

\vskip 10pt

Thus, it remains to show that an irreducible self-dual $\varphi^{\#}$ exists, or equivalently (by the local Langlands correspondence for $\GL_5$) that there exists a self-dual supercuspidal representation of $\GL_5(F)$.
 By the Jacquet-Langlands correspondence, it is equivalent to showing that  $D_5^{\times}$ has a self-dual irreducible representation of dimension $> 1$ (where $D_5$ is a division algebra of degree 5).  The group $D_5^{\times}$ has a standard decreasing filtration
\[  D_5^{\times} \supset D_5^{(1)} \supset D_5^{(2)} \supset .... \]
by open compact subgroups so that, for $i \geq 1$, the successive quotients $D_5^{(i)}/ D_5^{(i+1)}$ are equal to the additive group of a finite field of characteristic 2 (the degree 5 extension of the residue field of $F$). Since any irreducible representation
of $D_5^{\times}$ factors through the finite group $D_5^{\times}/ D_5^{(i+1)}$ for some $i$, we are reduced to showing that $D_5^{\times}/ D_5^{(i+1)}$ has non-central elements of order 2. This is certainly the case, as one can readily see by examining the elements in $D_5^{(i)}/D_5^{(i+1)}$.
\end{proof}

\vskip 10pt

We may now focus on the non-primitive $\phi$'s.
We first note the following lemma.

\vskip 5pt

\begin{Lem}\label{L:quartic}
Suppose that an irreducible  symplectic parameter $\phi$ is of the form $\phi = \Ind_{W_K}^{W_F} \chi$ with $K/F$ a quartic non-Galois extension which does not contain a quadratic subfield. Then one can find a quadratic extension $E/F$ such that $\phi = \Ind_{W_E}^{W_F} \sigma$ for some $\sigma$. Note that this situation is possible only when $p=2$ because if $p\neq 2$, any quartic extension contains an intermediate field.
\end{Lem}
\begin{proof}
Consider the restriction of $\phi$ to $W_K$. By Frobenius reciprocity and the irreducibility of $\phi$,
$\phi|_{W_K}$  contains the character $\chi$ with multiplicity one. Since $\phi$ is a symplectic parameter, the line affording the character $\chi$ is an isotropic line, and thus $\phi|_{W_K}$ must also contain the character
\[  \chi' = \simi(\phi)|_{W_K} \cdot \chi^{-1} \ne \chi, \]
so that
\[  \phi|_{W_K} = \chi \oplus  \chi' \oplus V \]
with $\dim V = 2$ and $\det V = \chi \cdot \chi'$.
\vskip 10pt

On the other hand, the fact that $K/F$ has no quadratic subfield implies that the double coset space $W_K\backslash W_F/W_K$ has size $2$. Thus Mackey's lemma implies that
\[  \phi|_{W_K} = \chi \oplus \Ind_{W_K \cap \tau^{-1}W_K \tau}^{W_K}  \chi^{\tau}, \]
where $\tau$ is any element of $W_F \smallsetminus W_K$.  The latter summand must contain $\chi'$ and so by Frobenius reciprocity, we have
\[  \chi'|_{W_K \cap \tau^{-1}W_K\tau} =  \chi^{\tau}|_{W_K \cap \tau^{-1}W_K\tau}. \]
Hence,
\[   \phi|_{W_K} = \chi \oplus \left( \chi' \cdot \Ind_{W_K \cap \tau^{-1}W_K \tau}^{W_K} 1 \right). \]
\vskip 10pt

Now let $L$ be the compositum of $K$ and $K^{\tau}$, so that $L/K$ is a cubic extension since
$W_L = W_K \cap \tau^{-1}W_K \tau$. We have:
\vskip 5pt

\begin{enumerate}[$\bullet$]
\item  $L/K$ is non-Galois. If not, then $\phi|_{W_K} = \chi \oplus \chi' \cdot (1 \oplus \mu \oplus \mu^2)$ where $\mu$ is a cubic character of $\Gal(L/K) \cong \Z/3\Z$. This would imply that $\chi \cdot \chi'= \det V = {\chi'}^2$, so that $\chi' = \chi$. This is a contradiction.
\vskip 5pt

\item If $M$ is the Galois closure of $L/K$, so that $\Gal(M/K) \cong S_3$, then $M$ is the Galois closure of $K/F$ so that $\Gal(M/F) \cong S_4$. Indeed, on one hand, $M$ is a degree 24 extension of $F$. On the other hand, the Galois closure of $K/F$ has degree $\leq 24$ and must contain $L$ and hence $M$. This shows that $M$ is the Galois closure of $K/F$ and $\Gal(M/F) \cong S_4$.
\end{enumerate}

\vskip 5pt

Now the sign character $\epsilon$ of $\Gal(M/F) \cong S_4$ determines a quadratic extension $E/F$. Moreover, we have
\[  \phi|_{W_K} =  \chi \oplus \chi'  \oplus \chi' \cdot V_0 \]
where $V_0$ is the (unique) irreducible 2-dimensional representation of $\Gal(M/K) \cong S_3$. Note that $\det V_0$ is the sign character of $\Gal(M/K)$, which is $\epsilon|_{W_K}$. Since $V = \chi' \cdot V_0$ and $\chi \cdot \chi' = \det V$, we deduce that
\[  \chi' = \chi \cdot \epsilon|_{W_K}, \]
which implies that
\[  \phi \otimes \epsilon = \Ind_{W_K}^{W_F} (\chi \cdot \epsilon|_{W_K}) = \Ind_{W_K}^{W_F} \chi' = \phi. \]
This shows that
\[  \phi  \cong \Ind_{W_E}^{W_F} \sigma \]
for some $\sigma$. This finishes the proof of the lemma.
\end{proof}

\vskip 10pt

Now we are ready to prove Proposition \ref{P:parameter}.
\vskip 5pt

\begin{proof}[{\bf Proof of Proposition \ref{P:parameter}}]
When $p\neq 2$, this has been shown by Vigneras \cite{V}. We shall argue generally below. Let us recall that the similitude character $\simi(\phi)$ occurs in $\wedge^2 \phi$, and
\begin{equation} \label{E:stdphi}
   \simi(\phi)^{-1} \cdot \bigwedge^2 \phi = 1 \bigoplus std(\phi).
\end{equation}

\vskip 5pt

The case of primitive $\phi$'s have been handled by
 Lemma \ref{L:primitive}  and  Proposition \ref{P:dipendra}.  Thus, we may assume that $\phi$ is not primitive below. Hence by the above lemma, we may suppose that
\[
    \phi = \Ind_{W_E}^{W_F} \sigma,
\]
for some quadratic extension $E/F$ with $\Gal(E/F)  = \langle \tau \rangle$ and some irreducible 2-dimensional representation $\sigma$. Then note that
\begin{equation} \label{E:detphiF}
  \simi(\phi)^2 = \det\phi = \det \sigma|_{F^{\times}}
\end{equation}
as characters of $F^{\times}$.  Moreover, we have
\begin{equation} \label{E:wedge}
  \bigwedge^2 \phi = \Ind_{W_E}^{W_F} \det\sigma \bigoplus M(\sigma),
\end{equation}
where $M(\sigma)$ is the multiplicative induction of $\sigma$ to $W_F$, which is simply an extension of $\sigma \otimes \sigma^{\tau}$ from $W_E$ to $W_F$ and sometimes called the Asai lift of $\sigma$. (See \cite[\S 7]{Pr} for this notion.) Now note that in  $\wedge^2 \phi|_{W_E}$, any $1$-dimensional character occurs with multiplicity at most $1$, except
for the character $\det \sigma$ which may occur with multiplicity 2.
To see this, observe that
since $\sigma$ is irreducible, any 1-dimensional character $\chi$ occurs in $M(\sigma)|_{W_K} = \sigma \otimes \sigma^{\tau}$ with multiplicity at most 1, because
\[
\dim   \Hom_{W_E}(\sigma \otimes \sigma^{\tau},\chi)=\dim \Hom_{W_E}(\sigma^{\tau},\chi\otimes{\sigma}^\vee)\leq 1
\]
by Schur's lemma. Also neither $\det\sigma$ nor $\det \sigma^{\tau}$ occurs in $\sigma \otimes \sigma^{\tau}$; if either one of them does, then we would have
\[  \sigma^{\tau} \cong  \sigma^{\vee} \otimes \det \sigma \cong \sigma, \]
which is a contradiction to the assumption that $\sigma$ is not $\tau$-invariant.
Hence in $\wedge^2 \phi|_{W_E}$, any 1-dimensional character occurs with multiplicity  at most 1, except perhaps for the character $\det \sigma$, which may occur with multiplicity 2 in
$(\Ind_{W_E}^{W_F} \det\sigma)|_{W_E}$.

\vskip 10pt

Now if $\simi(\phi)$  occurs in the first summand on the RHS of  (\ref{E:wedge}), then
\[  \simi(\phi)|_{W_E} = \det\sigma, \]
and we are in situation (III) of the proposition. Moreover, (\ref{E:stdphi}) and (\ref{E:wedge}) imply that
$std(\phi)$ contains $\omega_{E/F}$.
\vskip 5pt

Henceforth, we suppose that $\simi(\phi)$ occurs in $M(\sigma)$.
Then we have
\begin{equation}\label{E:sigma^tau}
     \sigma^{\tau} \cong \sigma^{\vee} \otimes \simi(\phi)|_{W_E}=\sigma\otimes(\simi(\phi)|_{W_E}/\det\sigma).
\end{equation}
Thus we have
\begin{equation} \label{E:Msigma}
  \simi(\phi)^{-1} \cdot M(\sigma)|_{W_E} = 1 \bigoplus Ad(\sigma).
\end{equation}
Since the LHS is $\tau$-invariant, $\tau$ must permute the 1-dimensional constituents of $Ad(\sigma)$ if there are any. These 1-dimensional constituents are precisely those quadratic characters $\omega$ of $W_E$ with respect to which $\sigma$ is dihedral, \ie such that $\sigma \otimes \omega \cong \sigma$.
If $\sigma$ is dihedral, $Ad(\sigma)$ contains  1 or 3 quadratic characters. In either case, we see that at least one of these quadratic characters must be fixed by $\tau$.
If we denote this $\tau$-invariant quadratic character by $\omega_0$, then (\ref{E:Msigma}) shows that $\simi(\phi)^{-1} \cdot M(\sigma)$ contains an extension of $\omega_0$ to $W_F$, which may be a quadratic or quartic character.
\vskip 5pt

For each of those cases, we will show below that if $\sigma$ is primitive, either (II) or (III) holds, and if it is dihedral, then $\omega_0$ can actually extend only to a quadratic character and for this case (III) happens.
\vskip 5pt

\noindent{\bf Case 1-a:} $\sigma$ is primitive and $\chi := \simi(\phi)|_{W_E}/\det\sigma$ is not quadratic. Note that by (\ref{E:sigma^tau})
\[   \sigma^{\tau} = \sigma \otimes \chi, \]
and so
\[  \det \sigma^{\tau} = \chi^2 \cdot \det \sigma \ne \det \sigma. \]
Moreover, equations  (\ref{E:stdphi}), (\ref{E:wedge})  and (\ref{E:Msigma}) together imply that $std(\phi)$ decomposes as the sum of an irreducible 2-dimensional representation and an irreducible 3-dimensional representation. Thus we are in situation (II) of the proposition.
\vskip 10pt

\noindent{\bf Case 1-b:}  $\sigma$ is primitive and $\chi := \simi(\phi)|_{W_E}/\det\sigma$ is quadratic. By (\ref{E:sigma^tau}),  we see that $\det \sigma^{\tau} = \det\sigma$, so that $\det \sigma$ extends to $W_F$. Since $\simi(\phi)|_{W_E} =  \chi \cdot \det\sigma$, we deduce that $\chi$ also extends to $W_F$ and one also has
\[  \simi(\phi)^2 = \det\sigma|_{F^{\times}}  \cdot \chi|_{F^{\times}} \quad \text{as characters on $F^{\times}$.} \]
Now the identity (\ref{E:detphiF})
implies that $\chi$ is trivial when restricted to $F^{\times}$. Thus $\chi$ extends to a quadratic character of $W_F$ and determines a quadratic extension $E'/F$. Moreover, $\phi \otimes \omega_{E'/F}\cong \phi$, since
\[  \phi \otimes \omega_{E'/F} = \Ind_{W_E}^{W_F} (\sigma \otimes \chi) = \Ind_{W_E}^{W_F} \sigma^{\tau} = \phi. \]
Thus $\wedge^2 \phi$ contains both $\simi(\phi)$ and $\simi(\phi) \cdot \omega_{E'/F}$, so that $\wedge^2\phi|_{W_{E'}}$ contains $\simi(\phi)|_{W_{E'}}$ with multiplicity two.
Hence, we conclude that
\[
    \phi \cong \Ind_{W_{E'}}^{W_F} \sigma' \quad \text{with}\quad \simi(\phi)|_{W_{E'}} = \det \sigma',
\]
so that we are in situation (III).

 \vskip 10pt

\noindent{\bf Case 2-a:}  $\sigma$ is dihedral and $\omega_0$ extends to a quadratic character $\omega_{E'/F}$. In  this case, it is clear that $\phi \otimes \omega_{E'/F} \cong \phi$, so that
\[  \phi = \Ind_{W_K}^{W_F} \sigma' \]
for some irreducible 2-dimensional $\sigma'$. Now, $\wedge^2 \phi$ contains the characters $\simi(\phi)$ and $\simi(\phi) \cdot \omega_{E'/F}$ and thus $\wedge^2 \phi|_{W_{E'}}$ contains $\simi(\phi)|_{W_{E'}}$ with multiplicity two. As we observed above, the only character which may occur with multiplicity two in $\wedge^2 \phi|_{W_{E'}}$ is $\det \sigma'$. Hence we have
\[  \simi(\phi)|_{W_{E'}} = \det \sigma' \]
and we are in situation (III) of the proposition.

 \vskip 10pt

\noindent{\bf Case 2-b:} $\sigma$ is dihedral and $\omega_0$ extends to a quartic character. In this case, the quadratic extension $K/E$ determined by $\omega_0$ is cyclic quartic over $F$ and
$E$ is the unique quadratic subfield of $K/F$. Set $\Gal(K/F) = \langle \tau \rangle$,
so that $\Gal(K/E) = \langle \tau^2 \rangle$. Then
\[   \phi = \Ind_{W_E}^{W_F} \sigma =  \Ind_{W_K}^{W_F} \chi. \]
Since $\sigma^{\tau} = \sigma^{\vee} \otimes \simi(\phi)|_{W_E}$, we see that
\[  \Ind_{W_K}^{W_E}( \chi^{\tau}) = \Ind_{W_K}^{W_E} (\chi^{-1} \cdot \simi(\phi)|_{W_K}). \]
This implies that
\[  \chi^{\tau} \cdot \chi = \simi(\phi)|_{W_K} \quad \text{or} \quad  \chi^{\tau^3} \cdot \chi = \simi(\phi)|_{W_K}\]
so that  $\chi \cdot \chi^{\tau}$ is $\tau$-invariant. But this implies that $\chi^{\tau^2} = \chi$, which contradicts the irreducibility of $\sigma$. Hence, $\omega_0$ cannot extend to a quartic character of $W_F$.

\vskip 10pt

Therefore we have thus shown that the situations (I), (II) and (III) of the proposition encompass all the possibilities for $\phi$. Moreover, from the behavior of $std(\phi)$, it is clear that these three situations are mutually exclusive. By the theorem by Koch \cite{Ko} mentioned right before the proposition,
 only (III) can occur if $p \ne 2$. So the only thing we are left with is to show that all of the three possibilities actually happen if $p=2$. We have constructed examples of primitive $\phi$ in Proposition \ref{P:dipendra}. Also it is easy to see that (III) can be achieved. Thus it remains to construct examples of situation (II).
\vskip 5pt

For this, let $E/F$ be quadratic extension with $\Gal(E/F) = \langle \tau \rangle$.
Let $\rho$ be a primitive 2-dimensional representation of $W_F$, and let $\chi$ be a character of $W_E$ such that $(\chi^{\tau}/\chi)^2 \ne 1$. Note that $\rho|_{W_E}$ is still primitive, since $Ad(\rho)$ is irreducible and thus $Ad(\rho)|_{W_E}$ cannot contain a 1-dimensional constituent and is thus irreducible also. Now consider the 4-dimensional representation
\[  \phi = \rho \otimes \Ind_{W_E}^{W_F} \chi = \Ind_{W_E}^{W_F} \rho|_{W_E} \otimes \chi. \]
This is an irreducible representation because, if we let $\sigma = \rho|_{W_E} \cdot \chi$,
\[  \sigma^{\tau} = \rho|_{W_E} \cdot \chi^{\tau} = \sigma \cdot (\chi^{\tau}/\chi)  \ne \sigma.\]
Here to show $\neq$, we have used the assumption that $(\chi^{\tau}/\chi)^2 \ne 1$.
\vskip 5pt

To show that $\phi$ is symplectic, we need to show that $\phi$ preserves a nondegenerate symplectic form up to scaling.
But
\[  \wedge^2 \phi = \wedge^2 \rho \otimes Sym^2(\Ind_{W_E}^{W_F} \chi)  \bigoplus
Sym^2\rho \otimes \wedge^2 \Ind_{W_E}^{W_F} \chi. \]
Now $Sym^2(\rho)$ is irreducible, so that the second summand contains no 1-dimensional character.
Moreover,
\[ Ad(\Ind_{W_E}^{W_F} \chi)  = \omega_{E/F} \oplus \Ind_{W_E}^{W_F} \chi^{\tau}/\chi, \]
and the second summand is irreducible. Thus we see that $\wedge^2 \phi$ contains a unique 1-dimensional character, namely $\det \rho \cdot \chi|_{F^{\times}}$ (regarded as a character of $F^{\times}$). In other words, $\phi$ preserves a unique symplectic form up to scaling (necessarily nondegenerate since $\phi$ is irreducible) and  the similitude character $\simi(\phi)$ satisfies
\[  \simi(\phi)|_{W_E} =  \det \rho |_{W_E} \cdot \chi \cdot \chi^{\tau} = \det \sigma \cdot (\chi^{\tau}/\chi) \ne \det\sigma. \]
Thus $\phi$ satisfies all the requirements of situation (II).
\end{proof}

\vskip 15pt

\section{\bf Sizes of $L$-packets of $\Sp_4$}\label{S:L-packet}


In this section, we determine the sizes of the L-packets of $\Sp_4$. More precisely, given an L-parameter $\phi: WD_F \longrightarrow \GSp_4(\C)$,  we describe the size of the $L$-packet of $\Sp_4$ associated to $\varphi =std(\phi)$ in terms of Galois theoretic properties of $\phi$.
\vskip 10pt

First of all, it is quite elementary to see that the possible sizes of the $L$-packet are given by
\[
    \# \mathcal{L}_{\varphi}=
    \begin{cases}
    1,2, 4\text{ or }8, &\text{ if $p\neq 2$;}\\
    1,2, 4,8\text{ or }16, &\text{ if $p= 2$}.
    \end{cases}
\]
When $p\ne 2$, this follows immediately from Proposition \ref{P:exact} and the fact that
$\#A_{\phi} \leq 2$ and $\#I(\phi) =1$, $2$ or $4$ since there are only three quadratic characters. To deal with the case $p=2$, it is probably easier to
work with the parameter $\varphi$. One has the following general statement \cite[Corollary 6.6]{GP}:
\vskip 5pt

\begin{Lem} \label{L:size}
Let $\varphi: WD_F \longrightarrow \SO_N(\C)$ be an admissible homomorphism with $N$ odd, and regard it as an $N$-dimensional representation of $WD_F$ with isotypic decomposition
\[  \varphi = \bigoplus_i n_i \cdot M_i. \]
Then
\[  \pi_0(Z_{\SO_N}(Im(\varphi))) = (\Z/2\Z)^{r-1} \]
where
\[  r = \# \{i: \text{$M_i^{\vee} \cong M_i$ and $M_i$ is  an orthogonal representation} \}. \]
\end{Lem}
\vskip 5pt

The lemma immediately implies:
\vskip 5pt

\begin{Cor} \label{C:size}
Let $\varphi: WD_F \longrightarrow \SO_5(\C)$ be an L-parameter for $\Sp_4(F)$. Then
\[
    \# \mathcal{L}_{\varphi}=
    \begin{cases}
    1,2, 4\text{ or }8, &\text{ if $p\neq 2$;}\\
    1,2, 4,8\text{ or }16, &\text{ if $p= 2$}.
    \end{cases}
\]
\end{Cor}

 \vskip 5pt

Of course, the corollary does not show that all possibilities for $\#L_{\varphi}$ occur.
For the rest of the section, we will show that all of them do occur, together with precise information on when each case happens in terms of the Galois theoretic properties of any $\phi\in\Phi(\GSp_4)$ for which $\varphi=std(\phi)$.\\
Our result is then the analog of the following well-known result for $\SL_2$ (cf. \cite{Sh}).
\vskip 10pt

\begin{Prop}
Let $\sigma: WD_F \longrightarrow \GL_2(\C)$ be an irreducible representation  and $\sigma_0 = Ad(\sigma): WD_F \longrightarrow \SO_3(\C)$ the associated discrete series L-parameter for $\SL_2$. Then
\[  \# L_{\sigma_0} = \begin{cases}
1, \text{  if $\sigma$ is primitive or $\sigma$ is non-trivial on $\SL_2(\C)$;} \\
2, \text{  if $\sigma$ is dihedral w.r.t. 1 quadratic field;} \\
4, \text{  if $\sigma$ is dihedral w.r.t. 3 quadratic fields.} \end{cases} \]

\vskip 5pt

\noindent In particular, the size of the $L$-packet $\# L_{\sigma_0}$ is equal to the number of characters $\omega$ such that $\sigma \otimes \omega \cong \sigma$ and the non-trivial such characters are precisely those which occur in $\sigma_0= Ad(\sigma)$. Moreover, the third case above happens iff $\sigma=\Ind_{W_E}^{W_F}\rho$ for some quadratic $E/F$ with $\rho^c/\rho$ a non-trivial quadratic character for the non-trivial element $c\in\Gal(E/F)$.
\end{Prop}
\vskip 10pt

We now consider the question of determining the size of the L-packet of $\Sp_4$ associated to the L-parameter $\varphi = std(\phi)$.  By our main theorem, this is given by
\[
     N(\phi) := \text{size of the component group $A_{std(\phi)}$}.
\]
If $\phi$ is a discrete series parameter, then $std(\phi)$ is a discrete series parameter for $\Sp_4$ and is the multiplicity free direct sum of orthogonal submodules as a 5-dimensional representation, in which case by Lemma \ref{L:size}, we know
\[
    \log_2 N(\phi) = (\text{the number of irreducible constituents of $std(\phi)$}) -1.
\]
This gives a convenient way of calculating $N(\phi)$ for discrete series parameters.
Recall that we have defined
\[
    I(\phi) = \{ \text{quadratic characters $\chi$ of $W_F$ with $\phi \otimes \chi \cong \phi$} \}
\]
and have shown that
\[
    N(\phi) = \#L_{\phi}\cdot\#I(\phi)= \begin{cases}
    \#I(\phi), \text{  if $L_{\phi}$ is a singleton;} \\
    2 \cdot \# I(\phi), \text{  otherwise.} \end{cases}
\]
In some cases (especially for non-discrete-series parameters), it will be easier to compute $N(\phi)$ by directly determining $I(\phi)$. Note that $I(\phi)$ is an elementary abelian 2-group and we shall frequently specify $I(\phi)$ by writing down a set of generators for it. For example, we shall write $I(\phi) = \langle \chi_1, \chi_2,\dots,\chi_k \rangle$ to indicate that it is generated by the $\chi_i$'s.
\vskip 5pt

The following lemma says that one can determine $I(\phi)$ by regarding $\phi$ as an L-parameter for $\GL_4$ and is useful for our purposes below:
  \vskip 5pt

\begin{Lem}\label{L:restriction}
Let $\phi \in \Phi(\GSp_4)$ and consider the natural inclusion
\[  \iota:  \GSp_4(\C) \hookrightarrow \GL_4(\C). \]
Then
\[
I(\phi)=\{\text{quadratic characters $\chi$}: \iota\circ \phi\cong(\iota\circ\phi)\otimes\chi\}.
\]
\end{Lem}
\vskip 5pt

\begin{proof}
It is obvious that the LHS is contained in the RHS. Conversely, suppose that
$\chi$ lies in the RHS. By \cite[Lemma 6.1]{GT}, one knows that the natural map
\[  \Phi(\GSp_4) \longrightarrow \Phi(\GL_4) \times \Phi(\GL_1) \]
given by
\[  \phi \mapsto (\iota \circ \phi, \simi(\phi))\]
is injective. However, since $\chi$ lies in the RHS, we have:
\[  (\iota \circ (\phi \otimes \chi),  \simi(\phi \otimes \chi)) = ((\iota \circ \phi)\otimes \chi,  \simi(\phi)) = (\iota \circ \phi, \simi(\phi)). \]
Hence, we conclude by injectivity that $\phi \otimes \chi = \phi$ , as desired.
\end{proof}
\vskip 15pt

The rest of this paper is devoted to the determination of the number $N(\phi)$ and the group $I(\phi)$.
\vskip 10pt

\noindent{\bf \underline{Discrete Series Parameters}}

\vskip 10pt

We begin with the most nondegenerate cases, namely those treated in Proposition \ref{P:parameter}.
\vskip 10pt

\begin{Thm}
Suppose that $\phi: W_F \longrightarrow \GSp_4(\C)$ is an irreducible representation with similitude character $\simi(\phi)$.
\vskip 5pt

\begin{enumerate}[\upshape (I)]
\item If $\phi$ is primitive, i.e. as in situation {\upshape (I)} of Proposition \ref{P:parameter}, then
$N(\phi) =1$.
\vskip 5pt

\item If $\phi$ is as in situation {\upshape (II)} of Proposition \ref{P:parameter}, then $N(\phi) = 2$ and $I(\phi) = \langle \omega_{E/F} \rangle$.

\vskip 5pt

\item If $\phi = \Ind_{W_E}^{W_F} \sigma$ is as in situation {\upshape (III)} of Proposition \ref{P:parameter}, then we have the following two cases, each of which is further divided into several subcases.
\vskip 5pt

\begin{enumerate}[(a)]
\item Suppose that $\sigma^{\tau} \ne \sigma \otimes \chi$ for any character $\chi$. Then, we have:
\[  N(\phi) = \begin{cases}
4, \text{ if $\sigma$ is dihedral w.r.t. a quadratic $K/E$ such that $K/F$ is biquadratic;} \\
2, \text{  otherwise.} \end{cases} \]
and respectively
\[  I(\phi) = \begin{cases}
\langle  \omega_{E/F} , \omega_{E'/F} \rangle, \text{  with $E'\ne E$ a quadratic extension in the biquadratic $K/F$}; \\
\langle \omega_{E/F} \rangle.
\end{cases} \]

\vskip 5pt
\noindent More precisely, we have:
\begin{enumerate}[({a}1)]

\item if $\sigma$ is primitive, then $N(\phi) = 2$.

\item if $\sigma = Ind_{W_K}^{W_E} \rho $ with $Gal(K/F) = \Z/4\Z$, then $N(\phi) = 2$.

\item if $\sigma = Ind_{W_K}^{W_E} \rho $ with $Gal(K/F) = \Z/2\Z \times \Z/2\Z$, then $N(\phi) = 4$.

\item if $\sigma = Ind_{W_K}^{W_E} \rho $ with $K/F$ non-Galois, then
\[  N(\phi) = \begin{cases}
2, \text{  if $\sigma^{\tau}|_{W_K} \cdot \rho$ does not extend to $W_F$;} \\
4, \text{  otherwise.} \end{cases} \]
Moreover, $\sigma^{\tau}|_{W_K} \cdot \rho$ is extendable to $W_F$ iff $\sigma$ is dihedral with respect to a $K'/E$ with $K'/F$ biquadratic, in which case we are reduced to situation (a3) by replacing $K$ by $K'$.
\end{enumerate}

\vskip 10pt

\item Suppose that $\sigma^{\tau} = \sigma \otimes \chi$ (so that $\chi$ is necessarily quadratic).

\vskip 5pt

\begin{enumerate}[({b}1)]
\item  If $\chi^{\tau} \ne \chi$, then $\sigma$ is dihedral with respect to $\chi \cdot \chi^{\tau}$, which determines a biquadratic extension $K/F$. In this case,
\[  N(\phi) = 4 \]
and
\[
    I(\phi) = \langle \omega_{E/F}, \omega_{E'/F}\rangle,
\]
where $E' \ne E$ is a quadratic extension of $F$ contained in $K$.
\vskip 5pt

\item If $\chi^{\tau} = \chi$, then there is a character $\lambda$ of $W_E$ such that $\chi = \lambda^{\tau}/\lambda$ and $\sigma \otimes \lambda^{-1}$ extends to a representation $\pi$ of $W_F$.  In other words, $\chi$ determines a biquadratic extension $K/F$ with quadratic subfields $E$, $E'$ and $E''$, and
\[  \sigma = \pi|_{W_E} \otimes \lambda \quad  \text{and}  \quad \phi = \pi \otimes \Ind_{W_E}^{W_F} \lambda. \]
Then we have:
\[  N(\phi) = \begin{cases}
4, \text{  if $\pi$ is primitive;} \\
8, \text{  if $\pi$ is dihedral w.r.t. 1 quadratic field $M/F$;} \\
16, \text{  if $\pi$ is dihedral w.r.t. 3 quadratic fields $M_i/F$,} \end{cases} \]
and respectively
\[  I(\phi) = \begin{cases}
\langle \omega_{E/F}, \omega_{E'/F} \rangle; \\
\langle \omega_{E/F}, \omega_{E'/F}, \omega_{M/F} \rangle; \\
\langle \omega_{E/F}, \omega_{E'/F}, \omega_{M_1/F}, \omega_{M_2/F}  \rangle.
 \end{cases} \]
\end{enumerate}
Moreover situation (b2) can occur only when $p=2$.

\end{enumerate}
\end{enumerate}
\end{Thm}
\vskip 5pt

\begin{proof}
The statements (I) and (II) are clear from Proposition \ref{P:parameter} and so we focus on (III).
Hence, we are in situation (III) of Proposition \ref{P:parameter}, so that
\[  \phi = \Ind_{W_E}^{W_F} \sigma \quad \text{with} \quad \simi(\phi)|_{W_E} = \det \sigma. \]
In this case, we have
\[  std(\phi) = \omega_{E/F} \oplus \simi(\phi)^{-1} \cdot M(\sigma). \]
So we need to determine how many irreducible constituents $M(\sigma)$ has.
As in the statement of the proposition, we consider the two cases (a) and (b).
\vskip 5pt

\noindent\underline{(a) $\sigma^{\tau} \ne \sigma \otimes \chi$ for any character
$\chi$}: In this case we see that
\[  M(\sigma)|_{W_E} = \sigma^{\tau} \otimes \sigma \]
does not contain any 1-dimensional constituent, so that
\[  M(\sigma)|_{W_E} = 4 \quad \text{or} \quad 2+2, \]
where the RHS means that the representation in question is either irreducible or the sum of 2 two-dimensional irreducible constituents. We consider the following different cases:
\vskip 5pt

\begin{enumerate}[({a}1)]
\item $\sigma$ primitive. Then we claim that $\sigma^{\tau} \otimes \sigma$ is irreducible. If not, then \[ \sigma^{\tau}\otimes \sigma = V_1 \oplus V_2, \quad \text{with $\dim V_i = 2$,} \]
which implies that
\[  \wedge^2 (\sigma^{\tau} \otimes \sigma) = \wedge^2 V_1 \oplus \wedge^2V_2 \oplus (V_1 \otimes V_2) \]
contains 1-dimensional constituents. However, we also have
\[  \wedge^2 (\sigma^{\tau} \otimes \sigma) = (\wedge^2 \sigma^{\tau} \otimes Sym^2(\sigma)) \bigoplus
(Sym^2(\sigma^{\tau}) \otimes \wedge^2 \sigma)  = 3+3. \]
This gives the desired contradiction. Hence $M(\sigma)$ is irreducible in this case and $N(\phi) = 2$.

\vskip 5pt

\item  $\sigma = \Ind_{W_K}^{W_E} \rho$ with $\Gal(K/F) = \Z/4\Z = \langle \tau \rangle$.
In this case, one has
\[ M(\sigma)|_{W_E} =  \sigma^{\tau} \otimes \sigma = \Ind_{W_K}^{W_E} \rho \cdot \rho^{\tau} \bigoplus \Ind_{W_K}^{W_E} \rho^{\tau} \cdot \rho^{\tau^2}  \]
and $\tau$ switches these two components. Note that those two components are non-isomorphic, because otherwise we would have $\rho \cdot \rho^{\tau}= \rho^{\tau} \cdot \rho^{\tau^2}$
or $(\rho^{\tau} \cdot \rho^{\tau^2})^{\tau^2}$,
which immediately implies $\rho^{\tau^2}=\rho$, thus contradicting the irreducibility of $\sigma$. So $M(\sigma)$ is irreducible and $N(\phi) = 2$.

\vskip 5pt

\item  $\sigma = \Ind_{W_K}^{W_E} \rho$ with $\Gal(K/F) = \Z/2\Z \times \Z/2\Z$.
In this case, let $\tau_1$ and $\tau_2$ be the two elements of $\Gal(K/F)$ which projects to the non-trivial element $\tau \in \Gal(E/F)$. One has:
\[  M(\sigma)|_{W_E} = \sigma^{\tau} \otimes \sigma = \Ind_{W_K}^{W_E} \rho \cdot \rho^{\tau_1} \bigoplus
 \Ind_{W_K}^{W_E} \rho \cdot \rho^{\tau_2} \]
and $\tau$ fixes each of the two components on the RHS. So
\[  M(\sigma) = 2+2 \]
and $N(\phi) = 4$.
\vskip 5pt

\item  $\sigma = \Ind_{W_K}^{W_E} \rho$ with $K/F$ non-Galois.
In this case, if $L$ denotes the composite of $K$ and $K^\tau$, then $L$ is the Galois closure of $K/F$
and $\Gal(L/F)$ is a non-abelian group of order 8. Since it contains the Klein 4-group $\Gal(L/E)$, it must in fact be the dihedral group of order $8$. In any case,
\[  \sigma^{\tau} \otimes \sigma|_{W_K} = \rho \cdot \sigma^{\tau}|_{W_K} \bigoplus \rho' \cdot \sigma^{\tau}|_{W_K}. \]
where $\rho' \ne \rho$ is the conjugate of $\rho$ by $W_E$.
It is not difficult to see that
\[   \sigma^{\tau} \ne \sigma \otimes \chi \Longleftrightarrow \rho'/\rho \ne \omega_{L/K} \Longleftrightarrow
\text{$\sigma^{\tau}|_{W_K}$ is irreducible}. \]
Hence, $\sigma^{\tau} \otimes \sigma|_{W_K} = 2+2$ and we have:
\[  N(\phi) = \begin{cases}
2, \text{ if $\rho \cdot \sigma^{\tau}|_{W_K}$  does not extend to $W_F$}; \\
4, \text{ otherwise.} \end{cases} \]

Let us examine this condition of extendability of $\sigma^{\tau}|_{W_K} \cdot \rho$ to $W_F$ in greater detail. For this, we first consider the issue of extendability to $W_E$.  Clearly, $\sigma^{\tau}|_{W_K} \cdot \rho$ can be extended to $W_E$ iff
\[  (\rho'/\rho) \otimes \sigma^{\tau}|_{W_K} \cong \sigma^{\tau}|_{W_K}. \]
This is equivalent to saying that $\rho'/\rho$ defines a biquadratic $M/E$ (different from $L/E$) containing quadratic subfields $K$, $K'$ and $K''$ with respect to which $\sigma$ is dihedral and such that $\sigma^{\tau}$ is dihedral with respect to $K'$ (without loss of generality). In that case, $K'$ is necessarily Galois over $F$.
In other words, the mere requirement that $\sigma^{\tau}|_{W_K} \cdot \rho$ be extendable to $W_E$ already forces $\sigma$ to be dihedral with respect to $K'/E$ such that $K'/F$ is Galois, so that we are reduced to situation (a2) or (a3). Depending on whether $K'/F$ is cyclic or biquadratic, we then obtain $N(\phi) = 2$ or $4$ respectively. Hence, our result can be stated as follows.
\[
    N(\phi) = \begin{cases}
        2, \text{ if $\rho \cdot \sigma^{\tau}|_{W_K}$  does not extend to $W_E$}; \\
        2, \text{ if $\rho \cdot \sigma^{\tau}|_{W_K}$  extends to $W_E$ but not to $W_F$};\\
        4, \text{ if $\rho \cdot \sigma^{\tau}|_{W_K}$  extends to $W_F$.}
    \end{cases}
\]
In the second case, we can be reduced to situation (a2), and in the third case, we can be reduced to (a3) by re-choosing the quartic extension $K$.
\end{enumerate}
\vskip 10pt

\noindent\underline{(b) $\sigma^{\tau} = \sigma \otimes \chi$ for some $\chi$}:
Since $\det \sigma^{\tau} = \det\sigma$, we see that
\[  \chi^2 = 1\quad \text{and} \quad  \sigma = \sigma \otimes \chi\chi^{\tau}. \]
We thus have the following cases:
\vskip 5pt
\begin{enumerate}[({b}1)]
\item  $\chi^{\tau} \ne \chi$. In this case, we need to show that $N(\phi) = 4$. Since $\sigma$
is dihedral with respect to the quadratic extension $K$ of $E$ defined by $\chi \cdot \chi^{\tau}$, it is not primitive. Moreover, since $\omega_{K/E}= \chi \cdot \chi^{\tau}$ is $\tau$-invariant and trivial on $F^{\times}$, it extends to a quadratic character of $W_F$. Hence
the quartic field $K/F$ is biquadratic. Let $\tau_1$ and $\tau_2$ be the two elements of $\Gal(K/F)$ which project to the element $\tau \in \Gal(E/F)$ and let $c = \tau_1 \cdot \tau_2$ so that $\Gal(K/E) = \langle c \rangle$. Writing
\[  \sigma = \Ind_{W_K}^{W_E} \rho, \]
the fact that $\sigma^{\tau} = \sigma \otimes \chi$ implies (without loss of generality) that
\[   \rho^{\tau_1}/ \rho  = \chi|_{W_K}. \]
Hence we have
\[  (\rho^c/\rho)^{\tau_1} = \rho^{\tau_1c}/\rho^{\tau_1} = \rho^c \cdot \chi|_{W_K} / \rho \cdot \chi|_{W_K}  = \rho^c/\rho. \]
Now we have
\[  \left(\simi(\phi)^{-1} \cdot M(\sigma) \right)|_{W_E} = \chi \oplus \chi^{\tau}  \oplus
\chi \cdot \Ind_{W_K}^{W_E} \rho^c/\rho, \]
and the first two summands on the RHS are exchanged by $\tau$.
If $\rho^c/\rho$ is not quadratic, then we see that
\[  M(\sigma) = 2+ 2 \quad \text{and} \quad N(\phi) = 4. \]
On the other hand, suppose that $\rho^c/\rho$ is quadratic. Then $\rho^c/\rho$ is fixed by $c$ and $\tau_1$ and hence by $\Gal(K/F)$. If $\mu$ is an extension of $\rho^c/\rho$ to $W_E$, then
\[  \left( \simi(\phi)^{-1} \cdot M(\sigma)\right)|_{W_E} = \chi \oplus \chi^{\tau}  \oplus \chi \mu  \oplus \chi^{\tau} \mu, \]
and $\tau$ exchanges the last two summands  as well, since $\mu$ is $\tau$-invariant. So we again have
\[  M(\sigma) = 2+2  \quad \text{and} \quad N(\phi) = 4. \]
\vskip 10pt

\item $\chi^{\tau} = \chi$. In this case, the quadratic character $\chi$ extends to $W_F$. We claim that it must extend to a quadratic character of $W_F$ instead of quartic, or equivalently that $\chi|_{F^{\times}} =1$. Suppose on the contrary that $\chi$ extends to a quartic character $\tilde{\chi}$ with $\tilde{\chi}^2  = \omega_{E/F}$. Then we would have
\[  \phi = \Ind_{W_E}^{W_F} \sigma^{\tau} = \Ind_{W_E}^{W_F} \sigma \cdot \chi = \phi \cdot \tilde{\chi}, \]
and this would imply
\begin{equation} \label{E:wedge2}
  \wedge^2 \phi = \wedge^2 \phi \cdot \tilde{\chi}^2 = \wedge^2 \phi \cdot \omega_{E/F}. \end{equation}
But $\simi(\phi)^{-1} \cdot \wedge^2\phi|_{W_E}$ contains $\chi$ with multiplicity one, so that $\simi(\phi)^{-1} \cdot \wedge^2\phi$ contains precisely one extension of $\chi$ to $W_F$. Without loss of generality, one may suppose that $\simi(\phi)^{-1} \cdot \wedge^2\phi$ contains $\tilde{\chi}$ but not $\tilde{\chi}^{-1} = \tilde{\chi} \cdot \omega_{E/F}$. But this contradicts equation (\ref{E:wedge2}).
With this contradiction, we conclude that $\chi$ extends to a quadratic character of $W_F$, or equivalently that $\chi|_{F^{\times}} =1$. Thus $\chi$ determines a biquadratic extension of $F$ with quadratic subfields $E$, $E_1$ and $E_2$.

\vskip 10pt

Now since $\chi|_{F^{\times}} = 1$, there exists a character of $W_E$ such that
\[  \chi = \lambda^{\tau}/\lambda. \]
Then we see that
\[  (\sigma \cdot \lambda^{-1})^{\tau} = \sigma \cdot \lambda^{-1}. \]
So there exists an irreducible two dimensional representation $\pi$ of $W_F$ such that
\[  \sigma = \pi|_{W_E} \cdot \lambda, \]
and hence
\[  \phi = \pi \otimes \Ind_{W_E}^{W_F} \lambda. \]
But then
\[  \wedge^2 \phi =  \det \pi \cdot \lambda|_{F^{\times}} \cdot \omega_{E/F} \cdot \left( \omega_{E/F} \oplus \omega_{E_1/F} \oplus \omega_{E_2/F} \oplus Ad(\pi)\right). \]
So we conclude that
\[ N(\phi) = \begin{cases}
4, \text{  if $\pi$ (or equivalently $\sigma$)  is primitive;} \\
8, \text{  if $\pi$ is dihedral w.r.t. 1 quadratic field;} \\
16, \text{  if $\pi$ is dihedral w.r.t. 3 quadratic fields.}
\end{cases} \]
Note that in the last two cases, the quadratic fields with respect to which $\pi$ is dihedral are necessarily different from $E$, $E_1$ and $E_2$ (by the irreducibility of $\phi$). Moreover it is clear that these three cases can occur only when $p=2$.
\end{enumerate}
\end{proof}
\vskip 20pt

\begin{Rmk}
Let us note that for the $\GSp_4$-parameter $\phi$ in the above theorem, the $L$-packet $L_{\phi}$ consists of a unique supercuspidal representation $\pi$ of $\GSp_4(F)$, and $N(\phi)$ is the number of irreducible constituents of the restriction $\pi|_{\Sp_4(F)}$.
\end{Rmk}
\vskip 10pt

The following proposition determines $N(\phi)$ and $I(\phi)$ for the remaining discrete series parameters $\phi$.
\vskip 5pt

\begin{Prop}
Let $S_n$ denote the $n$-dimensional representation of $\SL_2(\C)$. We have:
\vskip 5pt
\begin{enumerate}[(i)]
\item If $\phi = \mu \boxtimes S_4$ with $\mu$ a 1-dimensional character of $W_F$, then
\[
    N(\phi) = 1
\]
and so, of course,
\[
    I(\phi)=\langle 1 \rangle.
\]
\vskip 5pt

\item If $\phi = \sigma \boxtimes S_2$ with $\sigma$ an irreducible  2-dimensional dihedral representation of $W_F$, then
\[  N(\phi) = \begin{cases}
2, \text{ if $\sigma$ is dihedral with respect to 1 quadratic field $E/F$;} \\
4, \text{  if $\sigma$ is dihedral with respect to 3 quadratic fields $E_i/F$,}  \end{cases} \]
and respectively
\[ I(\phi) = \begin{cases}
\langle \omega_{E/F} \rangle; \\
\langle \omega_{E_1/F}, \omega_{E_2/F} \rangle. \end{cases} \]

\vskip 5pt

\item If $\phi = \phi_1 \oplus \phi_2$ where $\phi_1 \ne \phi_2$ are irreducible 2-dimensional representations of $WD_F$ with $\det \phi_1 = \det \phi_2$, then we have two cases:
\vskip 5pt

\begin{enumerate}[(a)]
\item if $\phi_1 \ne \phi_2 \otimes \chi$ for any $\chi$, then
\[
     N(\phi) = 2 \cdot \#I(\phi)= \begin{cases}
                4, \text{  if $\phi_1$ and $\phi_2$ are dihedral w.r.t. the same quadratic $E/F$;} \\
                2, \text{  otherwise}, \end{cases} \]
and respectively
\[
    I(\phi)=\begin{cases}
    \langle \omega_{E/F} \rangle\\
    \langle1 \rangle.
    \end{cases}
\]

\vskip 5pt

\item if $\phi_1 = \phi_2 \otimes \chi$, with $\chi$ necessarily quadratic, then
 \[  N(\phi) = 2 \cdot \# I(\phi)= \begin{cases}
 4,  \text{  if $\phi_1$ is primitive or non-trivial on $\SL_2(\C)$;} \\
 8, \text{  if $\phi_1$ is dihedral w.r.t. 1 quadratic field $E/F$;} \\
 16, \text{  if $\phi_1$ is dihedral w.r.t. 3 quadratic fields $E_i/F$;} \end{cases} \]
 and respectively
 \[  I(\phi) =  \begin{cases}
 \langle \chi \rangle; \\
 \langle \chi, \omega_{E/F} \rangle; \\
 \langle \chi, \omega_{E_1/F}, \omega_{E_2/F} \rangle. \end{cases} \]
\end{enumerate}
\end{enumerate}
\end{Prop}

\vskip 10pt

\begin{proof}\quad\\
\noindent (i) It is easy to see that $std(\phi) = S_5$ is irreducible.

\vskip 5pt

\noindent (ii) This follows from
 \[ \det \sigma^{-1} \cdot  \wedge^2 \phi = Ad(\sigma) \bigoplus  S_3. \]
\vskip 5pt

 \noindent (iii) Consider the two cases separately.
 \vskip 5pt
 \begin{enumerate}[(a)]
\item If $\phi_1 \ne \phi_2 \otimes \chi$, then
 \[  \det \phi_1 \cdot std(\phi) = \det\phi_1  \bigoplus \phi_1 \otimes \phi_2, \]
 and $\phi_1 \otimes \phi_2 = 2+2$ or $4$. So $N(\phi) = 2$ or $4$. On the other hand, it is clear that
  \[  I(\phi) = \{ \text{quadratic characters $\chi$: $\phi_i \otimes \chi = \phi_i$} \} \]
  This gives the desired result.
  \vskip 5pt
 Observe in particular that it is not possible for $\phi_1$ and $\phi_2$ to be dihedral with respect to the same 3 quadratic fields. This is well-known and can in fact be checked directly. Namely, suppose that $\phi_i= \Ind_{W_E}^{W_F} \rho_i$ with $\rho_1^c/\rho_1 = \rho_2^c/\rho_2$ quadratic. Then $\rho_1/\rho_2$ is $\Gal(E/F)$-invariant and
\[  \det \phi_1 = \det \phi_2 \Longrightarrow \rho_1|_{F^{\times}} = \rho_2|_{F^{\times}}, \]
so that $\rho_1/\rho_2$ is quadratic and trivial on $F^{\times}$. Hence $\rho_1/\rho_2$ extends to a quadratic character $\omega$. But this implies that $\phi_1 = \phi_2 \otimes \omega$, which contradicts our assumption that $\phi_1 \ne \phi_2 \otimes \chi$ for any $\chi$.

  \vskip 10pt

  \item If $\phi_1 = \phi_2 \otimes \chi$, with $\chi^2 = 1$, then $\phi_1$ is not dihedral with respect to
  $\chi$. Moreover, it is clear that $I(\phi)$ is generated by $\chi$ and the quadratic characters $\omega$ such that $\phi_1 = \phi_1 \otimes \omega$. This gives the desired result.
\end{enumerate}
\end{proof}

\begin{Rmk}
In the above proposition, in situations (i) and (ii), the $\GSp_4$ $L$-packet $L_{\phi}$ consists of a unique non-supercuspidal representation. In situation (iii), it consists of two discrete series representations. Also in (ii), in order for $\phi$ to be symplectic, $\sigma$ is necessarily dihedral. (See \cite{GT2}.)
\end{Rmk}
\vskip 20pt

\noindent{\bf \underline{Non-Discrete Series Parameters}}
\vskip 10pt

The following proposition determines $N(\phi)$ and $I(\phi)$ for  the non-discrete series parameters
$\phi$ of $\GSp_4$. We omit the proof as it is quite simple. Indeed, one can easily determine $I(\phi)$ by using Lemma \ref{L:restriction}.

\vskip 10pt

\begin{Prop}
Consider the non-discrete-series parameters $\phi$ of $\GSp_4(F)$ which fall into 3 families according to
the smallest Levi subgroup of $\GSp_4(\C)$ through which $\phi$ factors.
\vskip 5pt

\begin{enumerate}[(i)]
\item If $Im(\phi)$ is contained in a Siegel parabolic (but not a Borel subgroup), so that  $\phi = \sigma \oplus (\sigma \otimes \chi)$ with $\sigma$ irreducible and $\simi(\phi) = \chi \cdot \det \sigma$, then we consider the following different cases:
\begin{enumerate}[(a)]
\item if $\chi^2 \ne 1$, then
\[  N(\phi) = \# I(\phi)= \begin{cases}
1, \text{  if $\sigma$ is primitive or non-trivial on $\SL_2(\C)$;} \\
2, \text{  if $\sigma$ is dihedral w.r.t. 1 quadratic field $E$;} \\
4, \text{  if $\sigma$ is dihedral w.r.t. 3 quadratic fields $E_i$;} \end{cases} \]
and respectively
\[  I(\phi) = \begin{cases}
\langle 1 \rangle\\
\langle \omega_{E/F} \rangle; \\
\langle \omega_{E_1/F} , \omega_{E_2/F} \rangle. \end{cases} \]

\vskip 5pt

\item if $\chi^2 = 1$, then
\[  N(\phi) = \begin{cases}
\#I(\phi), \text{  if $\chi\ne 1$;} \\
2 \cdot \#I(\phi), \text{  if $\chi=1$.} \end{cases} \]
Moreover,
\[  I(\phi) = \begin{cases}
\langle \chi \rangle, \text{ if $\sigma$ is primitive or non-trivial on $\SL_2(\C)$;} \\
\langle \chi, \omega_{E/F} \rangle, \text{  if $\sigma$ is dihedral w.r.t. 1 quadratic $E/F$;} \\
\langle \chi, \omega_{E_1/F} , \omega_{E_2/F} \rangle, \text{  if $\sigma$ dihedral w.r.t. 3 quadratic $E_i/F$,} \end{cases} \]
with $\chi$ suppressed if $\chi=1$.
\end{enumerate}
\vskip 10pt

\item If $Im(\phi)$ is contained in a Klingen parabolic (but not a Borel subgroup), so that $\phi = \chi \cdot (1 \oplus \sigma \oplus \det\sigma)$ with $\sigma$ irreducible and $\simi(\phi) = \chi^2 \cdot \det \sigma$, then
\[  N(\phi) = \# I(\phi) = \begin{cases}
2, \text{  if $\sigma$ is dihedral with respect to $E/F$ and $\det\sigma = \omega_{E/F}$;} \\
1, \text{  otherwise,} \end{cases} \]
and respectively
\[
    I(\phi)=\begin{cases}
    \langle \omega_{E/F} \rangle\\
    \langle 1 \rangle.
    \end{cases}
\]

\vskip 10pt

\item If $Im(\phi)$ is contained in a Borel subgroup, so that $\phi = \chi \cdot (\chi_1\chi_2\oplus \chi_1 \oplus \chi_2\oplus 1)$, with $\simi(\phi) = \chi^2\chi_1\chi_2$,  then the non-trivial quadratic characters in $I(\phi)$ are precisely the distinct non-trivial quadratic characters amongst $\chi_1$,
$\chi_2$ and $\chi_1\chi_2$. More precisely,
\[  N(\phi) = \# I(\phi) = \begin{cases}
4, \text{  if $\chi_1 \ne \chi_2$ are nontrivial quadratic;} \\
1, \text{  if none of $\chi_1$, $\chi_2$ and $\chi_1\chi_2$ is non-trivial quadratic;} \\
2, \text{  if exactly one of $\chi_1$, $\chi_2$ and $\chi_1\chi_2$ is  quadratic,}
     \end{cases}
\]
and respectively
\[  I(\phi) = \begin{cases}
        \langle \chi_1,\chi_2 \rangle\\
        \langle 1 \rangle\\
        \langle\omega\rangle,
        \end{cases}
\]
where $\omega$ is the unique quadratic character among $\chi_1$, $\chi_2$ or $\chi_1\chi_2$ in the last case.
\end{enumerate}

\end{Prop}
\vskip 5pt

\begin{Rmk}
In situation (b) in (i) above, if $\chi=1$, the $\GSp_4$ $L$-packet $L_{\phi}$ consists of two non-discrete series representations, namely in the notation of \cite{GT2}, $L_{\phi}=\{\pi_{gen}(\tau), \pi_{ng}(\tau)\}$ where $\tau$ is the irreducible admissible representation of $\GL_2(F)$ corresponding to $\sigma$. For all the other cases, $L_{\phi}$ consists of a unique non-discrete series representation.
\end{Rmk}

\vskip 20pt


\appendix
\section{\bf Restrictions of admissible representations of $\GSp_4(F)$ to $\Sp_4(F)$}


Recall that for each $\pi\in\Pi(\GSp_4)$ we define
\[
    \JH(\pi):=\{\text{constituents of } {\pi}|_{\Sp_{4}}\}.
\]
Then the following tables summarize the sizes of $\JH(\pi)$ for irreducible admissible representations $\pi$ of $\GSp_4(F)$. Those tables are obtained simply by translating what we have obtained for the ``Galois side'' in the previous section to the ``automorphic side". We use the notations of \cite{GT2} to describe $\pi$. (See Table 1 of \cite{GT2}.)

\begin{table}[h]
\centering \caption{Restriction from $\GSp_4(F)$ to $\Sp_4(F)$ (Supercuspidal)}
\label{maintable} \vspace{-3ex}
$$
\renewcommand{\arraystretch}{1.5}
 \begin{array}{|c|c|c|c|c|c|}
  \hline&&\multicolumn{2}{|c|}{\pi}
   &\mbox{$\#\JH(\pi)$}&\\\hline\hline

   &\multirow{2}{*}{a}&\multicolumn{2}{|c|}{\text{not a lift from}}
   &2, 4
   &\mbox{$p\neq 2$}\\ \cline{5-6}
   &&\multicolumn{2}{|c|}{\text{$\GSO(2,2)$ or $\GSO(4)$}}
   &1,2,4,8,16
   &\mbox{$p=2$}\\ \cline{3-6}

   &&&\multirow{2}{*}{$\tau_1\neq\tau_2\otimes\chi$}
   &2
   &\mbox{$\tau_1$ and $\tau_2$ are dihedral w.r.t. same quad. ext.}\\ \cline{5-6}
   \mbox{S.C.}&&\mbox{$\theta(\tau_1\boxtimes\tau_2)$}&
   &1
   &\mbox{otherwise}\\ \cline{4-6}
   &b, c&\mbox{or}&
   &2
   &\mbox{$\tau_1$ is primitive or twisted Steinberg.}\\ \cline{5-6}
   &&\mbox{$\theta(\tau_1^D\boxtimes\tau_2^D)$}&\mbox{$\tau_1=\tau_2\otimes\chi$}
   &4
   &\mbox{$\tau_1$ is dihedral w.r.t. 1 quad. ext.}\\ \cline{5-6}
   &&&
   &8
   &\mbox{$\tau_1$ is dihedral w.r.t. 3 quad. ext.}\\ \hline

\end{array}
$$
\end{table}

\begin{table}[h]
\centering \caption{Restriction from $\GSp_4(F)$ to $\Sp_4(F)$ (Discrete Series)}
\label{maintable} \vspace{-3ex}
$$
\renewcommand{\arraystretch}{1.5}
 \begin{array}{|c|c|c|c|c|c|}
  \hline&&\multicolumn{2}{|c|}{\pi}
   &\mbox{$\#\JH(\pi)$}&\\\hline\hline

   &\multirow{2}{*}{a}&\multicolumn{2}{|c|}{\multirow{2}{*}{$St(\chi,\tau)$}}
   &2
   &\mbox{$\tau$ is dihedral w.r.t. 1 quad. ext.}\\ \cline{5-6}
   &&\multicolumn{2}{|c|}{}
   &4
   &\mbox{$\tau$ is dihedral w.r.t. 3 quad. ext.}\\ \cline{3-6}

\mbox{D.S.}&\multirow{2}{*}{b}&\multicolumn{2}{|c|}{\multirow{2}{*}{$St(\tau, \mu)$}}
   &2
   &\mbox{$\tau=st_{\chi}$ with $\chi$ non-trivial quadratic}\\ \cline{5-6}
   &&\multicolumn{2}{|c|}{}
   &1
   &\mbox{otherwise}\\ \cline{3-6}

   &\mbox{c}&\multicolumn{2}{|c|}{St_{\PGSp_4}\otimes\mu}
   & 1
   &\\ \hline

\end{array}
$$
\end{table}

\begin{table}[h]
\centering \caption{Restriction from $\GSp_4(F)$ to $\Sp_4(F)$ (Non Discrete Series)}
\label{maintable} \vspace{-3ex}
$$
\renewcommand{\arraystretch}{1.5}
 \begin{array}{|c|c|c|c|c|c|}
  \hline&&\multicolumn{2}{|c|}{\pi}
   &\mbox{$\# \JH(\pi)$}&\\\hline\hline

   &&&
   &2
   &\mbox{$\tau$ is primitive or twisted Steinberg}\\ \cline{5-6}
   &&&\mbox{$\chi^2=1$}
   &4
   &\mbox{$\tau$ is dihedral w.r.t. 1 quad. ext.}\\ \cline{5-6}
   &\multirow{2}{*}{a}&\mbox{$J_{Q(Z)}(\chi,\tau)$}&
   &8
   &\mbox{$\tau$ is dihedral w.r.t. 3 quad. ext.}\\ \cline{4-6}
   &&\mbox{$\chi \ne 1$}&
   &1
   &\mbox{$\tau$ is primitive or twisted Steinberg.}\\ \cline{5-6}
   &&&\mbox{$\chi^2\neq 1$}
   &2
   &\mbox{$\tau$ is dihedral w.r.t. 1 quad. ext.}\\ \cline{5-6}
   &&&
   &4
   &\mbox{$\tau$ is primitive or twisted Steinberg}\\ \cline{3-6}

   \multirow{2}{*}{N.D.S.}&&\multicolumn{2}{|c|}{}
   &1
   &\mbox{$\tau$ is primitive or twisted Steinberg.}\\ \cline{5-6}
   &\mbox{b, c}  & \multicolumn{2}{|c|}{\pi_{gen}(\tau) \text{ or } \pi_{ng}(\tau)}
   &2
   &\mbox{$\tau$ is dihedral w.r.t. 1 quad. ext.}\\ \cline{5-6}
   &&\multicolumn{2}{|c|}{}
   &4
   &\mbox{$\tau$ is dihedral w.r.t. 3 quad. ext.}\\ \cline{3-6}

   &\multirow{2}{*}{d}&\multicolumn{2}{|c|}{\multirow{2}{*}{$J_{P(Y)}(\tau,\chi)$}}
   &2
   &\mbox{$\tau=\tau\otimes\omega_\tau$, $\omega_{\tau} \ne 1$}\\ \cline{5-6}
   &&\multicolumn{2}{|c|}{}
   &1
   &\mbox{otherwise}\\ \cline{3-6}

   &&\multicolumn{2}{|c|}{}
   &4
   &\mbox{$\chi_1 \ne \chi_2$ non-trivial quadratic.}\\ \cline{5-6}
   &\mbox{e}  & \multicolumn{2}{|c|}{J_B(\chi_1,\chi_2;\chi)}
   &1
   &\mbox{none of $\chi_1$, $\chi_2$ or $\chi_1\chi_2$ non-trivial quadratic}\\ \cline{5-6}
   &&\multicolumn{2}{|c|}{}
   &2
   &\mbox{otherwise}\\ \cline{3-6}
   \hline
\end{array}
$$
\end{table}

\newpage
\quad\\

\end{document}